\documentclass[12pt]{article}

\usepackage{multirow}
\usepackage{amsmath,amsthm,amsfonts,stmaryrd}
\usepackage{amssymb,latexsym}
\usepackage{graphicx,color}
\usepackage{subfig}
\usepackage{setspace} 
\usepackage{float} 
\newtheorem{remark}{Remark}[section]
\newtheorem{lemma}{Lemma}[section]
\newtheorem{theorem}{Theorem}[section]
\newtheorem{prop}{Proposition}[section]
\newtheorem{assumption}{Assumption}[section]
\newcommand{\set}[1]{\lbrace #1 \rbrace}

\newcommand{\jump}[1]{\llbracket #1 \rrbracket}
\newcommand{\mean}[1]{\{#1\}}

\newcommand{\abs}[1]{\lvert#1\rvert}
\newcommand{\norm}[1]{\lVert#1\rVert}

\newcommand{\A}[1]{A\Bigl( #1 \Bigr)}
\newcommand{\B}[1]{B\Bigl( #1 \Bigr)}

\newcommand{\Ah}[1]{A_h\Bigl( #1 \Bigr)}

\newcommand{\n}{\boldsymbol{n}}
\usepackage{booktabs}
\newcommand\ffb{\framebox[1.2\width]}

\newcommand\bu{\boldsymbol{u}}
\newcommand\bv{\boldsymbol{v}}

\newcommand\bx{\boldsymbol{x}}
\newcommand\by{\boldsymbol{y}}
\newcommand\br{\boldsymbol{r}}
\newcommand\bs{\boldsymbol{s}}
\newcommand\bn{\boldsymbol{n}}

\newcommand\bg{\boldsymbol{g}}

\newcommand\bF{\boldsymbol{f}}
\newcommand\bI{\boldsymbol{I}}

\newcommand\bS{\boldsymbol{S}}
\newcommand\bT{\boldsymbol{T}}
\newcommand\bP{\boldsymbol{P}}

\newcommand\0{\mathbf{0}}

\newcommand{\cF}{\mathcal{F}}
\newcommand\cC{\mathcal{C}}
\newcommand\cP{\mathcal{P}}
\newcommand\cT{\mathcal{T}}

\newcommand\bcQ{\boldsymbol{\mathcal{Q}}}
\newcommand\bcW{\boldsymbol{\mathcal{W}}}

\newcommand\bcK{\boldsymbol{\mathcal{K}}}
\newcommand\bcE{\boldsymbol{\mathcal{E}}}

\newcommand\bcX{\mathbb X}
\newcommand{\dd}{\texttt{d}}
\newcommand\wbsig{\widetilde{\bsig}}
\newcommand\wbr{\widetilde{\br}}
\newcommand\beps{\boldsymbol{\varepsilon}}
\newcommand\bsig{\boldsymbol{\sigma}}
\newcommand\btau{\boldsymbol{\tau}}

\newcommand\wbu{\widetilde{\bu}}

\newcommand\R{\mathbb{R}}

\renewcommand\H{\mathrm{H}}
\renewcommand\L{\mathrm{L}}

\renewcommand\O{\Omega}
\newcommand\DO{\partial\O}
\newcommand\G{\Gamma}

\newcommand\bdiv{\mathop{\mathbf{div}}\nolimits}

\newcommand\tD{\mathtt{D}}
\newcommand\tr{\mathop{\mathrm{tr}}\nolimits}

\renewcommand\sp{\mathop{\mathrm{sp}}\nolimits}
\newcommand\dist{\mathop{\mathrm{dist}}\nolimits}
\renewcommand\t{\mathtt{t}}

\newcommand\LO{\L^2(\O)}
\newcommand\HdivO{{\H(\mathbf{div},\O)}}
\newcommand\HsdivO{{\H^s(\mathbf{div},\O)}}
\newcommand\HsO{\H^s(\O)}
\newcommand\nxn{n\times n}

\newcommand\ws{\widehat{s}}
\newcommand\sigmar{(\bsig, \br)}
\newcommand\taus{(\btau, \bs)}
\newcommand\sigmarh{(\bsig_h, \br_h)}
\newcommand\taush{(\btau_h, \bs_h)}
\newcommand\taushc{(\btau_h^c, \bs_h)}
\newcommand\tausht{(\tilde{\btau}_h, \mathbf{0})}

\renewcommand\t{\mathtt{t}}
\newcommand\gap{\widehat{\delta}}

\newcommand\disp{\displaystyle}
\newcommand\qin{\qquad\hbox{in }}
\newcommand\qon{\qquad\hbox{on }}

\date{}
\title{Mixed discontinuous Galerkin approximation of the elasticity eigenproblem }
\author{
{\sc Felipe Lepe}\thanks{Departamento de Matem\'atica,  
Universidad del B\'io-B\'io, Casilla 5-C, Concepci\'on, Chile and CI$^2$MA, Universidad de Concepci\'on, Concepci\'on, Chile.
(\texttt{flepe@ing-mat.udec.cl}). This author was partially supported by  Proyecto Plan Plurianual 2016-2020  Universidad del B\'io-B\'io (Chile) and CONICYT fellowship (Chile).}
$,\,\,$
{\sc Salim Meddahi}\thanks{Departamento de Matem\'aticas, Facultad de Ciencias,
Universidad de Oviedo, Calvo Sotelo s/n, Oviedo, Espa\~na. (\texttt{salim@uniovi.es}). This author was partially supported by Spain's Ministry of Economy Project MTM2017-87162-P (Spain).} 
$,\,\,$
{\sc David Mora}\thanks{Departamento de Matem\'atica,  
Universidad del B\'io-B\'io, 
Casilla 5-C, Concepci\'on, Chile 
and CI$^2$MA, Universidad de Concepci\'on, Concepci\'on, Chile.
(\texttt{dmora@ubiobio.cl}). This author was partially supported by CONICYT-Chile
through FONDECYT project 1140791 (Chile) and DIUBB through project 171508 GI/VC, Universidad del B\'io-B\'io (Chile).}\\
\, and\,  
{\sc Rodolfo Rodr\'iguez}\thanks{CI$^2$MA, 
Departamento de Ingenier\'{\i}a Matem\'atica,
Universidad de Concepci\'on, 
Casilla 160-C, Concepci\'on, Chile.
(\texttt{rodolfo@ing-mat.udec.cl}). This author was partially supported by BASAL project CMM, Universidad de Chile (Chile).}.
}

\begin{document}
\maketitle
%\author{Felipe Lepe}
%\address{CI$^{\mathrm{2}}$MA, Departamento de Ingenier\'{\i}a Matem\'atica,
%Universidad de Concepci\'on, Casilla 160-C, Concepci\'on, Chile.}
%\email{flepe@ing-mat.udec.cl}
%\thanks{The first author was supported by a CONICYT fellowship (Chile).}
%\author{Salim Meddahi}
%\address{Departamento de Matem\'aticas, Facultad de Ciencias, 
%Universidad de Oviedo, Calvo Sotelo s/n, Oviedo, Spain.}
%\email{salim@uniovi.es}
%\thanks{The second author was supported by
%Spain's Ministry of Economy Project MTM2013-43671-P}
%\author{David Mora}
%\address{Departamento de Matem\'atica, Universidad del B\'io-B\'io,
%Casilla 5-C, Concepci\'on, Chile and Centro de Investigaci\'on en
%Ingenier\'ia Matem\'atica (CI$^{\mathrm{2}}$MA),
%Universidad de Concepci\'on, Concepci\'on, Chile.}
%\email{dmora@ubiobio.cl}
%\thanks{The third author was partially supported by CONICYT-Chile
%through FONDECYT project 1140791 (Chile) and by DIUBB through project 151408 GI/VC,
%Universidad del B\'io-B\'io, (Chile)}
%\author{Rodolfo Rodr\'{\i}guez}
%\address{CI$^{\mathrm{2}}$MA, Departamento de Ingenier\'{\i}a Matem\'atica,  Universidad de
%Concepci\'on, Casilla 160-C, Concepci\'on, Chile.}
%\email{rodolfo@ing-mat.udec.cl}
%\thanks{The fourth author was partially supported by BASAL project CMM,
%Universidad de Chile (Chile).}
\begin{abstract}
We introduce a discontinuous Galerkin method for the mixed formulation of the elasticity eigenproblem with reduced symmetry. The analysis of the resulting discrete eigenproblem does not fit in the standard spectral approximation framework since the underlying source operator is not compact and the scheme is nonconforming. We show that the proposed scheme provides a correct approximation of the spectrum and prove  asymptotic error estimates for the eigenvalues and the eigenfunctions. Finally, we provide several  numerical tests to illustrate the performance of the method and confirm the theoretical results.
\end{abstract}
\medskip

\medskip

\medskip

\pagestyle{myheadings}
\thispagestyle{plain}
\markboth{S. MEDDAHI}{A  discontinuous Galerkin method for an  
elasticity  eigenproblem}

%%%%%%%%%%%%%%%%%%%%%%
\section{Introduction}\label{section:1}
%%%%%%%%%%%%%%%%%%%%%%
%The elasticity eigenvalue problem has been studied from several years, since the engineers are interested in the response of some structures used in the construction of buildings, aircrafts, cars, ships, and even in some medical applications. The classic finite element method has been an important tool to approximate the vibration modes of an elastic structure. However, the discontinuous Galerkin (DG) method
We present a discontinuous Galerkin (DG) approximation of the linearized vibrations of an elastic structure. In many applications, the displacement field is not necessarily the variable of  primary interest. We consider here the dual-mixed formulation of the elasticity eigenproblem because it delivers a direct finite element approximation of the Cauchy stress tensor and it permits to deal safely with nearly incompressible materials.
 
A mixed finite element approximation of the eigenvalue elasticity problem with reduced symmetry has been analyzed in \cite{MMR}. It consists in a formulation that only maintains  the stress tensor as primary unknown, besides the rotation whose role is the weak imposition of the symmetry restriction. It is shown that a discretization based on the lowest order Arnorld-Falk-Winther element provides  a correct spectral approximation and quasi optimal asymptotic error estimates for the eigenvalues and the eigenfunctions.

The ability of DG methods handle  efficiently $hp$-adaptive strategies make them suitable for the numerical simulation of physical systems related to elastodynamics.  Our aim here is to introduce an interior penalty discontinuous Galerkin version for the H(div)-conforming finite element space employed in \cite{MMR}. The $k^{th}$-order of this method amounts to approximate the Cauchy stress tensor and the rotation by discontinuous finite element spaces of degree $k$ and $k-1$ respectively. We point out that an H(curl)-based interior penalty discontinuous Galerkin method has also been introduced in \cite{BuffaPerugia} for the Maxwell eigensystem. The DG approximation we are considering here may be regarded as its counterpart in the H(div)-setting. As in \cite{BuffaPerugia}, our analysis requires conforming meshes, but the DG method still permits one to employ different polynomial element  orders in the same triangulation. A further advantage of this DG scheme is that it allows to implement high-order elements in a mixed formulation by using standard shape functions. Let us remark that the DG method has also been analyzed in \cite{Antonietti} for the Laplace operator.
 
 It is well known that the underlying source operator corresponding to mixed formulations is generally not compact. In our case, this operator admits a non physical zero eigenvalue  whose eigenspace is  infinite dimensional. It is then essential to use a scheme that is safe from the pollution that may appear in the form of spurious eigenvalues interspersed among the physically relevant ones.  It turns out (cf. \cite{ActaBoffi, BoffiBrezziGastaldi(a)}) that, for mixed eigenvalue problems, the conditions guarantying the convergence of the source problem does not necessarily a correct spectral approximation (as it happens for compact operators \cite{BO}).

It has been shown in \cite{BuffaPerugia} that DG methods can also benefit from the general theory developed in \cite{DNR1,DNR2} to deal with the spectral numerical analysis of non-compact operators. We follow here the same strategy, combined with techniques from \cite{MMR, MMT}, to prove that our numerical scheme is spurious free. We also establish asymptotic error estimates for the eigenvalues and eigenfunctions. We treat with special care the analysis of the limit problem obtained when the Lam\'e coefficient tends to infinity.

We end this section with some of the notations that we will  use below. Given
any Hilbert space $V$, let $V^n$ and $V^{\nxn}$ denote, respectively,
the space of vectors and tensors of order $n$ $(n= 2, 3)$  with
entries in $V$. In particular, $\bI$ is the identity matrix of
$\R^{\nxn}$ and $\mathbf{0}$ denotes a generic null vector or tensor. 
Given $\btau:=(\tau_{ij})$ and $\bsig:=(\sigma_{ij})\in\R^{\nxn}$, 
we define as usual the transpose tensor $\btau^{\t}:=(\tau_{ji})$, 
the trace $\tr\btau:=\sum_{i=1}^n\tau_{ii}$, the deviatoric tensor 
$\btau^{\tD}:=\btau-\frac{1}{n}\left(\tr\btau\right)\bI$, and the
tensor inner product $\btau:\bsig:=\sum_{i,j=1}^n\tau_{ij}\sigma_{ij}$. 

Let $\O$ be a polyhedral Lipschitz bounded domain of $\R^n$ with
boundary $\DO$. For $s\geq 0$, $\norm{\cdot}_{s,\O}$ stands indistinctly
for the norm of the Hilbertian Sobolev spaces $\HsO$, $\HsO^n$ or
$\HsO^{\nxn}$, with the convention $\H^0(\O):=\LO$. We also define for
$s\geq 0$ the Hilbert space 
$\HsdivO:=\set{\btau\in\HsO^{\nxn}:\ \bdiv\btau\in\HsO^n}$, whose norm
is given by $\norm{\btau}^2_{\HsdivO}
:=\norm{\btau}_{s,\O}^2+\norm{\bdiv\btau}^2_{s,\O}$ and denote
$\HdivO:={\H^0(\mathbf{div},\O)}$.

Henceforth, we denote by $C$ generic constants independent of the discretization
parameter, which may take different values at different places.

%%%%%%%%%%%%%%%%%%%%%%%%%%%%
\section{The model problem}\label{section:2}
%%%%%%%%%%%%%%%%%%%%%%%%%%%%

Let $\O\subset \R^n$ ($n=2,3$) be an open bounded Lipschitz 
polygon/polyhedron representing an elastic body. We denote by $\bn$ the outward unit normal 
vector to  $\DO$ and assume that  $\DO=\Gamma_D\cup\Gamma_N$, with $\textrm{int}(\Gamma_D)\cap \mathrm{int}(\Gamma_N) = \emptyset$.  
The solid is supposed to be isotropic
and linearly elastic with  mass density $\rho$ and Lam\'e constants $\mu$
and $\lambda$. We assume that the structure is fixed at $\Gamma_D\neq \emptyset$ and free
of stress on $\Gamma_N$. 
We can combine the constitutive law 
\begin{equation*}\label{constitutive}
\cC^{-1}\bsig=\beps(\bu) \qin\O,
\end{equation*}
and the equilibrium equation 
\begin{equation}\label{motion}
\omega^2 \bu  =  \rho^{-1}\bdiv \bsig \qin \O,
\end{equation}
to eliminate either the displacement field $\bu$ or the Cauchy stress tensor $\bsig$ from the global spectral 
formulation of the elasticity problem. Here,  
$\beps(\bu):=\frac{1}{2}[\nabla\bu+(\nabla\bu)^{\t}]$ is  the linearized strain tensor,  and 
$\cC:\ \R^{\nxn}\to\R^{\nxn}$ is the Hooke operator, which is given in terms of the Lam\'e coefficients 
$\lambda$ and $\mu$ by
\[
\cC\btau
:=\lambda\left(\tr\btau\right)\bI + 2\mu\btau \qquad\forall\,\btau \in \R^{\nxn}.
\]
Opting for the elimination of the displacement $\bu$ and  maintaining the stress tensor 
$\bsig$ as a main variable leads to the following dual mixed formulation 
of the elasticity eigenproblem:  
Find $\bsig:\O\to \R^{\nxn}$ symmetric,  $\br:\O\to \R^{\nxn}$ skew symmetric 
and $\omega\in \mathbb{R}$ such that,
\begin{equation}\label{modelPb}
\begin{array}{rcll}
  -\nabla\left(\rho^{-1} \bdiv\bsig \right) &=& \omega^2 \left( \cC^{-1}\bsig + \br \right)  & \qin\O,
 \\
\bdiv\bsig &=&\0 & \qon\Gamma_D,
\\
\bsig\bn&=&\0 & \qon\Gamma_N.
\end{array}
\end{equation}
We notice that the additional variable $\br:=\frac{1}{2}\left[\nabla\bu-(\nabla\bu)^{\t}\right]$ is the rotation. 
It acts as a Lagrange multiplier for the symmetry restriction. We also point out that 
the displacement can be recovered and also post-processed at the discrete level by using identity \eqref{motion}.
\medskip

Taking into account that the Neumann boundary condition becomes essential in the mixed formulation,  we consider the closed subspace $\bcW$ of $\HdivO$ given by
\[
 \bcW:=\left\{\btau\in \HdivO:\quad \btau\bn=\0\text{ on }\Gamma_N\right\}.
\]
 The rotation $\br$ will be sought in the space
\[
\bcQ:=\set{\bs\in\LO^{\nxn}:\ \bs^{\t}=-\bs}.
\]

We introduce the symmetric bilinear forms 
\[
 \B{\sigmar, \taus}:= \int_{\O} \cC^{-1}\bsig:\btau + \int_{\O}\br:\btau + \int_{\O}\bs:\bsig
\]
and
\[
 \A{\sigmar, \taus} := \int_{\O}\rho^{-1} \bdiv \bsig \cdot \bdiv \btau + \B{\sigmar, \taus}
\]
and denote the Hilbertian product norm on $\HdivO\times \L^2(\O)^{\nxn}$ by
\[
 \norm{\taus}^2 := \norm{\btau}^2_{\HdivO} + \norm{\bs}^2_{0,\O}.
\]

The variational formulation of the eigenvalue problem \eqref{modelPb} is given  as follows 
in terms of  $\kappa:=1 + \omega^2$   (see \cite{MMR} for more details):
Find $\kappa \in\R$ and $\0\neq (\bsig,\br)\in\bcW\times \bcQ$ such that
\begin{equation}\label{varForm}
\A{(\bsig,\br),(\btau,\bs)} = \kappa \, \B{(\bsig,\br),(\btau,\bs)}\quad \forall (\btau,\bs)\in \bcW\times \bcQ.
\end{equation}

We notice that the bilinear form 
\[
(\bsig, \btau)_{\cC, \bdiv} := \int_{\O} \rho^{-1}\bdiv  \bsig \cdot \bdiv \btau +  \int_{\O} \cC^{-1} \bsig:\btau
\]
also defines an inner product on $\bcW$. Moreover,  the following well-known result establishes that the norm 
induced by $(\cdot, \cdot)_{\cC,\bdiv}$ is equivalent to $\norm{\cdot}_{\HdivO}$ uniformly  in the Lam\'e coefficient $\lambda$.
\begin{prop}\label{normEquiv}
There exist constants $c_2\geq c_1>0$ independent of $\lambda$ such that 
\[
c_1 \norm{\btau}_{\HdivO} \leq \norm{\btau}_{\cC,\bdiv} \leq c_2 \norm{\btau}_{\HdivO}\quad \forall \btau \in \bcW,
\]
where $\norm{\btau}_{\cC,\bdiv}:= \sqrt{(\btau,\btau)}_{\cC,\bdiv}$.
\end{prop}
\begin{proof}
The bound from above follows immediately from the fact that 
\begin{equation}
\label{invcCop}
\int_{\O}\cC^{-1}\bsig:\btau = \frac{1}{2\mu}\int_{\O}\bsig^{\tD}:\btau^{\tD} +\frac{1}{n(n\lambda + 2\mu)} \int_{\O}(\tr\bsig)(\tr\btau)
\end{equation}
is bounded by a constant independent of $\lambda$. The left inequality may be found, for example, in \cite[Lemma 2.1]{MMR}.
\end{proof}

As a consequence of Proposition \ref{normEquiv},  
there exists a constant $M>0$ independent of $\lambda$ such that 
\begin{equation}\label{boundA}
\left|\A{\sigmar, \taus}\right| \leq M \, \norm{\sigmar} \norm{\taus} \quad \forall \sigmar, \taus \in \bcW\times \bcQ.
\end{equation}
\begin{prop}\label{infsupA-cont}
 There exists a constant $\alpha>0$, depending on $\rho$, $\mu$ and $\O$ (but not on $\lambda$), such that 
 \begin{equation}\label{infsupa}
  \sup_{\taus\in \bcW\times \bcQ} \frac{\A{\sigmar, \taus}}{\norm{\taus}} \geq \alpha \norm{\sigmar}\quad \forall \sigmar \in 
  \bcW\times \bcQ.
 \end{equation}
\end{prop}
\begin{proof}
It follows from Proposition \ref{normEquiv} that  
\begin{equation*}\label{elip0}
\A{(\btau, \mathbf{0}), (\btau, \mathbf{0})} = (\btau,\btau)_{\cC,\bdiv}
\geq C_1^2 \norm{\btau}^2_{\HdivO} 
\qquad\forall\btau\in\bcW,
\end{equation*}
with $C_1>0$ independent of $\lambda$. 
On the other hand, there exists a constant $\beta>0$ depending only on $\O$ 
(see, for instance, \cite{BoffiBrezziFortin}) such that 
 \begin{equation}\label{inSupbeta}
  \sup_{\btau\in \bcW} \frac{\int_{\O}\bs:\btau}{\norm{\btau}_{\HdivO}} \geq \beta \norm{\bs}_{0,\O} \qquad 
\forall \bs\in \bcQ.  
 \end{equation}
Consequently, the Babu\v{s}ka-Brezzi theory  
shows  that, for any bounded linear form $L\in \mathcal{L}(\bcW\times \bcQ)$, the 
problem: find $\sigmar\in \bcW\times \bcQ$ such that 
\[
 \A{\sigmar, \taus} = L\big(\btau,\bs\big)\qquad \forall \taus \in \bcW\times \bcQ
\]
is well-posed, which proves  \eqref{infsupa}.
\end{proof}

We deduce from Proposition \ref{infsupA-cont} and from the symmetry of 
$A(\cdot, \cdot)$ that the operator $\bT: [\L^2(\O)^{\nxn}]^2 \to \bcW\times \bcQ$ 
defined for any $(\bF, \bg) \in [\L^2(\O)^{\nxn}]^2$, by 
\begin{equation}\label{charcT}
 \A{\bT(\bF, \bg), \taus} = \B{(\bF, \bg), \taus} \quad \forall \taus \in \bcW\times \bcQ
\end{equation}
is well-defined and symmetric with respect to $A(\cdot, \cdot)$. Moreover, there exists a constant $C>0$ independent of 
$\lambda$ such that 
\begin{equation}\label{bT}
\norm{\bT (\bF, \bg)} \leq C \norm{(\bF, \bg)}_{0,\O}\quad \forall (\bF, \bg) \in [\L^2(\O)^{\nxn}]^2.
\end{equation}
It is clear that  $(\kappa,\sigmar)$ 
is a solution of \eqref{varForm} if and only if
$\left( \eta=\frac{1}{\kappa}, \sigmar\right)$ is an eigenpair for $\bT$. Let
\begin{equation}\label{K}
\bcK:=\set{\btau\in\bcW:\ \bdiv\btau=\0\ \ \mbox{in }\O}.
\end{equation}
{}From the definition of $\bT$, it is clear that 
$\bT|_{\bcK\times\bcQ}:\bcK\times\bcQ\longrightarrow\bcK\times\bcQ$
reduces to the identity. Thus, $\eta=1$ is an eigenvalue of $\bT$ with eigenspace $\bcK\times \bcQ$. We introduce the orthogonal subspace to $\bcK\times\bcQ$ in
$\bcW\times\bcQ$ with respect to the bilinear form $B$,
\begin{equation*}
[\bcK\times\bcQ]^{\bot_{B}}
:=\left\{ (\bsig,\br)\in\bcW\times\bcQ:
\ \B{ (\bsig,\br), (\btau,\bs)}=0 \quad \forall (\btau,\bs)\in\bcK\times\bcQ\right\}.
\end{equation*}

\begin{lemma}
\label{L1}
The subspace $[\bcK\times\bcQ]^{\bot_{B}}$ is invariant for
$\bT$, i.e.,
\begin{equation*}\label{Tinvariant}
 \bT([\bcK\times\bcQ]^{\bot_{B}})
\subset[\bcK\times\bcQ]^{\bot_{B}}.
\end{equation*}
Moreover, we have the direct and stable decomposition
\begin{equation}\label{split0}
\bcW\times\bcQ = [\bcK\times\bcQ]\oplus[\bcK\times\bcQ]^{\bot_{B}}.
\end{equation}
\end{lemma}
\begin{proof}
See Lemma 3.3 and Lemma 3.4 of \cite{MMR}.
\end{proof}
We deduce from Lemma \ref{L1} that there exists a unique projection $\bP:\, \bcW\times\bcQ\to \bcW\times\bcQ$
with range $[\bcK\times\bcQ]^{\bot_{B}}$ and kernel $\bcK\times\bcQ$
associated to the splitting \eqref{split0}.

Let us consider the elasticity problem posed in 
$\O$ with a volume load in $\L^2(\O)^n$ and with homogeneous Dirichlet and Neumann boundary conditions on $\Gamma_D$ and  $\Gamma_N$, respectively. According to \cite{D,grisvard}, there exists $\ws\in(0,1)$ that depends on $\O$, $\lambda$ and $\mu$ such that the displacement field that solves this problem belongs to $\H^{1+s}(\O)^n$ for all $s\in (0,\ws)$.
The following result shows that $\bP$ and $\bT\circ\bP$ are regularizing operators.

\begin{lemma}\label{reg}
For all $s\in (0, \ws)$,
$
\bP(\bcW \times \bcQ) \subset \HsO^{\nxn}\times \HsO^{\nxn}$ and  
$\bT(\bP(\bcW \times \bcQ)) \subset \{\HsO^{\nxn}\times \HsO^{\nxn}:\, \bdiv \btau \in \H^1(\O)^n\}$.
Moreover, there exists a constant $C>0$  such that 
\begin{equation}\label{reg1}
\norm{\bP \taus}_{\HsO^{\nxn}\times \HsO^{\nxn}} \leq C \norm{\bdiv \btau}_{0,\O}\quad \forall \taus \in \bcW\times \bcQ
\end{equation}
and
\begin{equation}\label{reg2}
\norm{\bT\circ \bP \taus}_{\HsdivO\times \HsO^{\nxn}} \leq C \norm{\bdiv \btau}_{0,\O}\quad \forall \taus \in \bcW\times \bcQ.
\end{equation}
\end{lemma}
\begin{proof}
Estimate \eqref{reg1} is proved in \cite[Lemma 3.2]{MMR} and \eqref{reg2} follows as a consequence of \eqref{reg1}, see \cite[Proposition~3.5]{MMR}. 
\end{proof}
We point out that, in principle, the exponent $\ws$ and the constant $C$ in \eqref{reg1} depend on the Lam\'e coefficient $\lambda$. However, we know that \eqref{reg1} also holds true when $\lambda=+\infty$ (see the Appendix). Hence, it is natural to expect \eqref{reg1} to be satisfied uniformly in $\lambda$. However, to the best of authors' knowledge, such a result is not available in the literature. For this reason, from now on we make the following assumption.
\begin{assumption}\label{assumpt1}
There exist $\ws\in (0,1)$ and $\widehat{C}_0>0$ independent of $\lambda$ such that 
\begin{equation*}\label{as1}
\norm{\bP \taus}_{\H^{s}(\O)^{\nxn}\times \H^{s}(\O)^{\nxn}} \leq \widehat{C}_0 \norm{\bdiv \btau}_{0,\O}\quad \forall \taus \in \bcW\times \bcQ,\quad \forall s\in (0, \ws).
\end{equation*}
\end{assumption}
This would immediately imply the existence of $\widehat{C}_1>0$ independent of $\lambda$ such that 
\begin{equation*}\label{assumptreg2}
\norm{\bT\circ \bP \taus}_{\HsdivO\times \HsO^{\nxn}} \leq \widehat{C}_1 \norm{\bdiv \btau}_{0,\O}\quad \forall \taus \in \bcW\times \bcQ,\quad \forall s\in (0, \ws).
\end{equation*}
The next result gives the spectral characterization for the solution operator $\bT$.
\begin{prop}\label{specT}
 The spectrum $\sp(\bT)$ of $\bT$ decomposes as follows 
 \[
  \sp(\bT) = \set{0, 1} \cup \set{\eta_k}_{k\in \mathbb{N}}
 \]
where $\set{\eta_k}_k\subset (0,1)$ is a real sequence of  
finite-multiplicity eigenvalues of  $\bT$ which converges to 0.  The ascent of each of
these eigenvalues is $1$ and the corresponding eigenfunctions lie in $\bP(\bcW \times \bcQ)$.
Moreover, $\eta=1$ is an infinite-multiplicity eigenvalue of $\bT$ with associated eigenspace 
$\bcK\times\bcQ$ and  $\eta=0$ is not 
an eigenvalue. 
\end{prop}
\begin{proof}
See \cite[Theorem 3.7]{MMR}.
\end{proof}

We end this section by providing a bound of the resolvent $\big( z\bI - \bT \big)^{-1}$. 
\begin{prop}\label{specT1}
	If $z \notin \sp(\bT)$, 
	there exists a constant $C>0$ independent of $\lambda$ and $z$ such that 
	\begin{equation*}\label{resolvent}
	\norm{\big(z\bI-\bT\big) ( \bsig,  \br)} \ge\, C \,  
	\dist\big(z,  \sp(\bT) \big)\,  \norm{( \bsig,  \br)} \quad \forall ( \bsig,  \br)\in \bcW\times \bcQ,
	\end{equation*}  
	where $\dist\big(z,  \sp(\bT) \big)$ represents the distance between $z$ and  
	the spectrum  of $\bT$ in the complex plane, which in principle depends on $\lambda$.
	
\end{prop}
 \begin{proof}
 See Proposition 2.4 in \cite{MMT}.
 \end{proof}

%%%%%%%%%%%%%%%%%%%%%%%%%%%%%%%%%%%%%%%%%%%%%%
\section{A discontinuous Galerkin discretization}\label{section:3}
%%%%%%%%%%%%%%%%%%%%%%%%%%%%%%%%%%%%%%%%%%%%%%

We consider  shape regular affine meshes $\mathcal{T}_h$ that subdivide the domain $\bar \Omega$ into  
triangles/tetrahedra $K$ of diameter $h_K$. The parameter $h:= \max_{K\in \cT_h} \{h_K\}$ 
represents the mesh size of $\cT_h$. Hereafter, given an integer $m\geq 0$ and a domain 
$D\subset \mathbb{R}^n$, $\cP_m(D)$ denotes the space of polynomials of degree at most $m$ on $D$.

We say that a closed subset $F\subset \overline{\Omega}$ is an interior edge/face if $F$ has a positive $(n-1)$-dimensional 
measure and if there are distinct elements $K$ and $K'$ such that $F =\bar K\cap \bar K'$. A closed 
subset $F\subset \overline{\Omega}$ is a boundary edge/face if
there exists $K\in \cT_h$ such that $F$ is an edge/face of $K$ and $F =  \bar K\cap \partial \Omega$. 
We consider the set $\cF_h^0$ of interior edges/faces and the set $\cF_h^\partial$ of boundary edges/faces.
We assume that the boundary mesh $\cF_h^\partial$ is compatible with the partition $\DO = \G_{D} \cup \G_{N}$, i.e., 
\[
\bigcup_{F\in \cF_h^D} F = \G_{D} \qquad \text{and} \qquad \bigcup_{F\in \cF_h^N} F = \G_N,
\]
where $\cF_h^D:= \set{F\in \cF_h^\partial; \quad F\subset \G_D}$ and 
$\cF_h^N:= \set{F\in \cF_h^\partial; \quad F\subset \G_N}$.
We denote   
\[
  \cF_h := \cF_h^0\cup \cF_h^\partial\qquad \text{and} \qquad \cF^*_h:= \cF_h^{0} \cup \cF_h^{N},
\]
and for any element $K\in \cT_h$, we introduce the set 
 \[
 \cF(K):= \set{F\in \cF_h;\quad F\subset \partial K} 
 \]
 of edges/faces composing the boundary of $K$.
 The space of piecewise polynomial functions of degree at most $m$ relatively to $\cT_h$ is denoted by  
\[
 \cP_m(\cT_h) :=\set{ v\in L^2(\O); \quad v|_K \in \cP_m(K),\quad \forall K\in \cT_h }. 
\]
For any $k\geq 1$, we consider the finite element spaces
\[
\bcW_h := \cP_k(\cT_h)^{\nxn}\qquad 
 \bcW_h^c := \bcW_h \cap \bcW 
 \qquad \text{and} \qquad \bcQ_h := \cP_{k-1}(\cT_h)^{\nxn} \cap \bcQ.
\]

Let us now recall some well-known properties of the Brezzi-Douglas-Marini (BDM) 
mixed finite element \cite{BDM}. For $t>1/2$, the tensorial version of the BDM-interpolation operator 
$\Pi_h: \H^t(\O)^{\nxn} \to \bcW_h^c$,  
satisfies the following classical error estimate, see \cite[Proposition 2.5.4]{BoffiBrezziFortinBook},
\begin{equation}\label{asymp0}
 \norm{\btau - \Pi_h \btau}_{0,\O} \leq C h^{\min(t, k+1)} \norm{\btau}_{t,\O} \qquad \forall \btau \in \H^t(\O)^{\nxn}, \quad t>1/2.
\end{equation}
For less regular tensorial fields we also have the following error estimate
\begin{equation}\label{asymp00}
 \norm{\btau - \Pi_h \btau}_{0,\O} \leq C h^t (\norm{\btau}_{t,\O} + \norm{\btau}_{\HdivO})  \quad \forall \btau \in \H^t(\O)^{\nxn}\cap \HdivO, \quad t\in (0, 1/2].
\end{equation}
Moreover, thanks to the commutativity property, if $\bdiv \btau \in \H^t(\O)^{n}$, then
\begin{equation}\label{asympDiv}
 \norm{\bdiv (\btau - \Pi_h \btau) }_{0,\O} = \norm{\bdiv \btau - \mathcal R_h \bdiv \btau }_{0,\O} 
 \leq C h^{\min(t, k)} \norm{\bdiv\btau}_{t,\O},
\end{equation}
where $\mathcal R_h$ is the $\LO^n$-orthogonal projection onto $\cP_{k-1}(\cT_h)^n$. Finally, 
we denote by $\mathcal S_h:\ \bcQ\to\bcQ_h$ the orthogonal
projector with respect to the $\LO^{\nxn}$-norm. 
It is well-known that, for any $t>0$, we have
\begin{equation}\label{asymQ}
 \norm{\br-\mathcal S_h\br}_{0,\O}
\leq C h^{\min(t, k)} \norm{\br}_{t,\O}
 \qquad\forall\br\in\H^t(\O)^{\nxn}\cap\bcQ.
\end{equation} 
For the analysis we need to decompose adequately the space $\bcW^c_h\times\bcQ_h$.
We consider,
\[
\bcK_h = \left\{\btau \in \bcW^c_h;\quad \bdiv \btau = 0  \right\} \subset \bcK.
\]
\begin{lemma}\label{Ph}
There exists a projection $\bP_h:\, \bcW_h^c \times \bcQ_h \to \bcW_h^c \times \bcQ_h$ with kernel 
$\bcK_h \times \bcQ_h$ such that for all $s\in (0,\ws)$, there exists a constant $C$ independent of $h$ and $\lambda$ such that
$$
\norm{(\bP - \bP_h)\sigmarh} \leq C\, h^s \norm{\bdiv \bsig_h}_{0,\O} \quad \forall \sigmarh\in 
\bcW_h^c \times \bcQ_h.
$$
\end{lemma}
\begin{proof}
See the proof of estimate (ii) of Lemma 4.2 from  \cite{MMR} 
\end{proof}
%\textcolor{red}{
%We denote as $\mathcal{G}_h$ the orthogonal complement for the space $\bcK_h$, which we define as follows:
%\[
%\mathcal{G}_h:=\{\sigmarh\in\bcW_h^c\times\bcQ_h: B(\sigmarh,(\btau_h,\bs_h))=0\quad\forall(\btau_h,\bs_h)\in\bcK_h\times\bcQ_h\},
%\] 
%with $B$ the bilinear form previously defined.}

For any $t\geq 0$, we consider the broken Sobolev space 
\[
 \H^t(\cT_h):=
 \set{\bv \in \L^2(\O)^n; \quad \bv|_K\in \H^t(K)^n\quad \forall K\in \cT_h}.
\]
For each $\bv:=\set{\bv_K}\in \H^t(\cT_h)^n$ and 
$\btau:= \set{\btau_K}\in \H^t(\cT_h)^{\nxn}$
the components $\bv_K$ and $\btau_K$  represent the restrictions $\bv|_K$ and $\btau|_K$. 
When no confusion arises, the restrictions of these functions will be written 
without any subscript. We will also need the space given on the skeletons of the triangulations $\cT_h$  by 
\[
 \L^2(\cF_h):= \prod_{F\in \mathcal{F}_h} \L^2(F).
 \]
Similarly, the components $\chi_F$ 
of $\chi := \set{\chi_F}\in \L^2(\cF_h)$  
coincide with the restrictions $\chi|_F$ and we denote
\[
 \int_{\cF_h} \chi := \sum_{F\in \cF_h} \int_F \chi_F\quad \text{and}\quad 
 \norm{\chi}^2_{0,\cF_h}:= \int_{\cF_h} \chi^2, 
 \qquad 
 \forall \chi\in \L^2(\cF_h).
\]
Similarly, $\norm{\chi}^2_{0,\cF^*_h}:= \sum_{F\in \cF^*_h}\int_{F} \chi_F^2$ for all $\chi\in \L^2(\cF^*_h):= \prod_{F\in \mathcal{F}^*_h} \L^2(F)$. 

From now on, $h_\cF\in \L^2(\cF_h)$ is the piecewise constant function 
defined by $h_\cF|_F := h_F$ for all $F \in \cF_h$ with $h_F$ denoting the 
diameter of edge/face $F$.

Given a vector valued function $\bv\in \H^t(\cT_h)^n$, with $t>1/2$, 
we define averages $\mean{\bv}\in \L^2(\cF_h)^n$ and jumps $\jump{\bv}\in \L^2(\cF_h)$ 
by
\[
 \mean{\bv}_F := (\bv_K + \bv_{K'})/2 \quad \text{and} \quad \jump{\bv}_F := \bv_K \cdot\n_K + \bv_{K'}\cdot\n_{K'} 
 \quad \forall F \in \cF(K)\cap \cF(K'),
\]
where $\n_K$ is the outward unit normal vector to $\partial K$. On the boundary of $\O$ we use the following 
conventions for averages and jumps:
\[
 \mean{\bv}_F := \bv_K  \quad \text{and} \quad \jump{\bv}_F := \bv_K \cdot\n 
 \quad \forall F \in \cF(K)\cap \DO.
\]

Similarly, for matrix valued functions $\btau\in \H^t(\cT_h)^{\nxn}$, we define $\mean{\btau}\in \L^2(\cF_h)^{\nxn}$ and 
$\jump{\btau}\in \L^2(\cF_h)^n$ by
\[
 \mean{\btau}_F := (\btau_K + \btau_{K'})/2 \quad \text{and} \quad \jump{\btau}_F := 
 \btau_K \n_K + \btau_{K'}\n_{K'} 
 \quad \forall F \in \cF(K)\cap \cF(K')
\]
and on the boundary of $\Omega$ we set 
\[
 \mean{\btau}_F := \btau_K  \quad \text{and} \quad \jump{\btau}_F := 
 \btau_K \n  
 \quad \forall F \in \cF(K)\cap \DO.
\]

Given $\btau \in \bcW_h$ we define $\bdiv_h \btau \in \L^2(\O)^n$ by 
$
\bdiv_h \btau|_{K} = \bdiv (\btau|_K)$ for all $K\in \cT_h
$ and 
endow $\bcW(h) := \bcW + \bcW_h$ with the seminorm
\[
 |\btau|^2_{\bcW(h)} := \norm{\bdiv_h \btau}^2_{0,\O} + \norm{h_{\cF}^{-1/2} \jump{\btau}}^2_{0,\cF^*_h}
\]
and the norm
\[
 \norm{\btau}^2_{\bcW(h)} := |\btau|^2_{\bcW(h)} + \norm{\btau}^2_{0,\O}.
\]
For the sake of simplicity, we will also use the notation 
\[
 \norm{(\btau, \bs)}^2_{DG} : =  \norm{\btau}^2_{\bcW(h)} + \norm{\bs}^2_{0,\O}.
\]
The following result will be used in the sequel to ultimately derive a method free of spurious modes. Since according to Proposition \ref{specT} the spectrum of $\bT$ lies in the unit disk $\mathbb{D}:=\{ z\in\mathbb{C}:\, |z|\leq 1\}$, we restrict our attention to this subset of the complex plane.

\begin{lemma}\label{TDG}
There exists a constant $C>0$ independent of $h$ and $\lambda$ such that for all  $z \in\mathbb{D}\setminus \sp(\bT)$ with $|z|\leq 1$, there holds
\[
\norm{(z \bI - \bT) \taus }_{DG} \geq C\dist\big(z,\sp(\bT)\big)|z|\, \norm{\taus}_{DG} \quad \forall \taus \in \bcW(h)\times \bcQ.
\]
\end{lemma}
\begin{proof}
We introduce
\[
(\bsig^*, \br^*):= \bT \taus \in \bcW\times \bcQ 
\]
and notice that 
\[
(z \bI - \bT)(\bsig^*, \br^*) = \bT (z \bI - \bT) \taus.
\]
By virtue of Proposition \ref{specT} and the boundedness of $\bT:\, [\L^2(\O)^{\nxn}]^2 \to \bcW\times \bcQ$ we have that 
\begin{multline*}
C\dist\big(z,\sp(\bT)\big) \norm{(\bsig^*, \br^*)}  \leq \norm{(z \bI - \bT)(\bsig^*, \br^*)}  \leq 
\norm{\bT (z \bI - \bT) \taus}\\ \leq \norm{\bT} \norm{(z \bI - \bT) \taus }_0 
\leq \norm{\bT} \norm{(z \bI - \bT) \taus }_{DG} .
\end{multline*}
Finally, by the triangle inequality, 
\begin{align*}
\norm{\taus}_{DG} \leq |z|^{-1} \norm{(\bsig^*, \br^*)} +& |z|^{-1} \norm{(z \bI - \bT) \taus }_{DG}\\
\leq &|z|^{-1}\left( 1+\dfrac{\norm{\bT}}{C\dist\big(z,\sp(\bT)\big)} \right) \norm{(z \bI - \bT) \taus }_{DG}\\
\leq & |z|^{-1}\left( \dfrac{C\dist\big(z,\sp(\bT)\big)+\norm{\bT}}{C\dist\big(z,\sp(\bT)\big)} \right)\norm{(z \bI - \bT) \taus }_{DG}.
\end{align*}
Hence, 
\begin{equation*}
C|z|\left(\frac{\dist\big(z,\sp(\bT)\big)}{\norm{\bT}+\dist\big(z,\sp(\bT)\big)}\right)\norm{\taus}_{DG}\leq \norm{(z \bI - \bT) \taus }_{DG}.
\end{equation*}
Since $\dist\big(z,\sp(\bT)\big)\leq |z|\leq 1$  and $\|\bT\|\leq C'$ (with $C'$ independent of $\lambda$), we derive from the above estimate that 
\begin{equation*}
\frac{C|z|}{1+C'} \dist\big(z,\sp(\bT)\big)\norm{\taus}_{DG} \leq \norm{(z \bI - \bT) \taus }_{DG}, 
\end{equation*}
and the result follows.
%We introduce
%\[
%(\bsig^*, \br^*):= \bT \taus \in \bcW\times \bcQ 
%\]
%and notice that 
%\[
%(z \bI - \bT)(\bsig^*, \br^*) = \bT (z \bI - \bT) \taus.
%\]
%By virtue of Proposition \ref{specT1} and the boundedness of $\bT:\, [\L^2(\O)^{\nxn}]^2 \to \bcW\times \bcQ$ we have that 
%\begin{multline*}
%C\dist\big(z,\sp(\bT)\big) \norm{(\bsig^*, \br^*)}  \leq \norm{(z \bI - \bT)(\bsig^*, \br^*)}  \leq 
%\norm{\bT (z \bI - \bT) \taus}\\ \leq \norm{\bT} \norm{(z \bI - \bT) \taus }_0 
%\leq \norm{\bT} \norm{(z \bI - \bT) \taus }_{DG} .
%\end{multline*}
%Finally, by the triangle inequality, 
%\begin{multline*}
%\norm{\taus}_{DG} \leq |z|^{-1} \norm{(\bsig^*, \br^*)} + |z|^{-1} \norm{(z \bI - \bT) \taus }_{DG}\\
%\leq |z|^{-1}\left( 1+\dfrac{\norm{\bT}}{C\dist\big(z,\sp(\bT)\big)} \right) \norm{(z \bI - \bT) \taus }_{DG}
%\end{multline*}
%and the result follows.
\end{proof}

\begin{remark}\label{roof}
If $E$ is a compact subset of $\mathbb{D} \setminus\sp(\bT)$, we deduce from Lemma \ref{TDG} that there exists a 
constant $C>0$ independent of $h$ and $\lambda$ such that, for all $z\in E$,
\begin{equation*}\label{resid}
\norm{\big(z \bI - \bT \big)^{-1}}_{\mathcal{L}(\bcW(h)\times \bcQ, \bcW(h)\times \bcQ)} \leq \frac{C}{\dist\big(E,\sp(\bT)\big)|z|}.
\end{equation*}
\end{remark}

Let us now introduce the discrete counterpart of \eqref{varForm}. Given a parameter $\texttt{a}_S>0$, we introduce  the symmetric bilinear form 
\begin{multline*}
 \Ah{\sigmar,\taus}:= 
 \int_{\O} \rho^{-1}\bdiv_h \bsig \cdot \bdiv_h \btau + \B{\sigmar, \taus}  \\
+ \int_{\cF^*_h} \texttt{a}_S h_{\cF}^{-1}\, \jump{\bsig}\cdot \jump{\btau}-\int_{\cF^*_h} 
\left( \mean{\rho^{-1}\bdiv_h\bsig} \cdot \jump{\btau} 
+ \mean{\rho^{-1}\bdiv_h\btau} \cdot \jump{\bsig} \right)
\end{multline*}
and consider the DG method: Find $\kappa_h\in \R$ and $0\neq\sigmarh\in \bcW_h\times \bcQ_h$ such that 
\begin{equation}\label{DGshort}
\Ah{\sigmarh, (\btau_h,\bs_h)} = \kappa_h \B{\sigmarh, (\btau_h,\bs_h)}\qquad \forall (\btau_h,\bs_h)\in \bcW_h\times \bcQ_h.
\end{equation}
We notice that, as it is usually the case for DG methods, the essential boundary condition is directly incorporated 
within the scheme. 

A straightforward application of the Cauchy-Schwarz  inequality shows that, for all $\sigmar, \taus \in  \H^t(\bdiv, \cT_h)\times \bcQ$ ($t>1/2$), there exists a constant $M^*>0$ independent of $h$ and $\lambda$ such that 
\begin{equation}\label{boundA1}
\left|\Ah{\sigmar,\taus}\right| \leq M^* \norm{\sigmar}^*_{DG}\, \norm{\taus}^*_{DG}, 
\end{equation}
where 
\[
\norm{\sigmar}^*_{DG}:= \Big( \norm{\sigmar}^2_{DG} +  
\norm{h_{\cF}^{1/2} \mean{\bdiv \bsig}}^2_{0,\cF^*_h}
\Big)^{1/2}.
\]
Moreover, we deduce from the discrete trace inequality (see \cite{DiPietroErn})
\begin{equation}\label{discTrace}
  \norm{h^{1/2}_{\cF}\mean{v}}_{0,\cF_h}\leq C \norm{v}_{0,\O}\quad \forall  v\in \cP_k(\cT_h),
 \end{equation}
that for all $\sigmar\in \H^t(\bdiv, \cT_h)\times \bcQ$ ($t>1/2$), and $(\btau_h,\bs_h) \in \bcW_h\times \bcQ_h$,
\begin{equation}\label{boundA2}
\left|\Ah{\sigmar,(\btau_h,\bs_h)}\right| \leq M_{DG} \norm{\sigmar}^*_{DG}\, \norm{(\btau_h,\bs_h)}_{DG},
\end{equation}
with $M_{DG}>0$ is independent of $h$ and $\lambda$.
\section{The DG-discrete source operator}\label{section:4}
%%%%%%%%%%%%%%%%%%%%%%%%%%%%%%%%%%%%%%%%%%%%%%%%%%%%%%%%

%\begin{prop}
%There exists a constant $C>0$ independent of $h$ and $\lambda$ such that, for all 
%$\sigmar\in \H^s(\bdiv, \O)\times \H^s(\O)^{\nxn}$, with $s>1/2$ and for all  
%$\taush\in \bcW_h\times \bcQ_h$,
%\[
%\Ah{\sigmar, \taush} \leq M_{DG} \norm{\sigmar}^*_{DG} \norm{\taush}_{DG}
%\]
%where
%\[
%\norm{\sigmar}^*_{DG}:= \Big( \norm{\sigmar}^2_{DG} +  
%\norm{h_{\cF}^{1/2} \mean{\bdiv \bsig}}^2_{0,\cF^*_h}
%\Big)^{1/2}
%\]
%\end{prop}

The following discrete projection operator from the DG-space $\bcW_h$ onto the $\H(\bdiv,\O)$-conforming mixed finite element space $\bcW^c$ is essential in the forthcoming analysis.

\begin{prop}\label{propC}
There exists a projection $\mathcal{I}_h:\, \bcW_h \to \bcW_h^c$ such that 
the norm equivalence 
\begin{equation}\label{equivN}
\underbar{C}\,  \norm{\btau}_{\bcW(h)}  \leq \Big( \norm{\mathcal{I}_h \btau}^2_{\HdivO} + 
 \norm{h_{\cF}^{-1/2} \jump{\btau}}^2_{0,\cF^*_h}
 \Big)^{1/2} \leq \bar{C} \norm{\btau}_{\bcW(h)} 
\end{equation}
holds true on $\bcW_h$ with constants $\underbar{C}>0$ and $\bar{C}>0$ independent of $h$. Moreover, we have that 
\begin{equation}\label{L2Ph}
 \norm{\bdiv_h (\btau- \mathcal{I}_h \btau)}^2_{0,\O} + \sum_{K\in \cT_h} h_K^{-2} \norm{\btau- \mathcal{I}_h \btau}^2_{0,K}
 \leq C_0  
 \norm{h_{\cF}^{-1/2} \jump{\btau}}^2_{0,\cF^*_h},
\end{equation}
 with $C_0>0$ independent of $h$.
\end{prop}
\begin{proof}
See \cite[Proposition 5.2]{MMT}.
\end{proof}
We can prove, with the aid of this result, that the bilinear form $A_h$ satisfies the following inf-sup condition that ensures the stability of our DG method.
\begin{prop}\label{infsupDh}
There exists a positive parameter $\textup{\texttt{a}}_S^*$ such that, for all $\textup{\texttt{a}}_S\geq \textup{\texttt{a}}_S^*$,
\begin{equation}\label{infsupABh}
  \sup_{\taush\in \bcW_h\times \bcQ_h} \frac{\Ah{\sigmarh, \taush}}{\norm{\taush}_{DG}} 
  \geq \alpha_{DG}   \norm{\sigmarh}_{DG} \quad \forall \sigmarh \in \bcW_h\times \bcQ_h
\end{equation} 
with $\alpha_{DG}>0$ independent of $h$ and $\lambda$.
\end{prop}
\begin{proof}
It is shown in \cite[Proposition 3.1]{MMT} that 
 there exists a constant $\alpha_A^{c}>0$ independent of $h$ and $\lambda$  such that 
 \begin{equation*}\label{infsupA-disceq}
  \sup_{\taush\in \bcW_h^c\times \bcQ_h} \frac{\A{\sigmarh, \taush}}{\norm{\taush}} 
  \geq \alpha_A^{c} \norm{\sigmarh}\quad \forall \sigmarh \in 
  \bcW_h^c\times \bcQ_h.
 \end{equation*}
 It follows that there exists an 
  operator $\Theta_h:\, \bcW^c_h\times \bcQ_h \to \bcW^c_h\times \bcQ_h$ satisfying
  \begin{equation}\label{cota1}
   \A{\sigmarh, \Theta_h \sigmarh} = \alpha_A^c\norm{\sigmarh}^2 \quad \text{and}\quad
   \norm{\Theta_h \sigmarh} \leq  \norm{\sigmarh}
  \end{equation}
  for all $\sigmarh\in \bcW^c_h\times \bcQ_h$.
  
  Given $\taush\in \bcW_h\times \bcQ_h$, the decomposition $\btau_h = \btau_h^c + \tilde\btau_h$, with   
  $\btau_h^c := \mathcal{I}_h \btau_h$ and $\tilde\btau_h := \btau_h - \mathcal{I}_h \btau_h$, and \eqref{cota1} yield   
  \begin{multline}\label{split}
  \Ah{\taush, \Theta_h \taushc+ \tausht} 
  = \alpha_A^c   \norm{\taushc}^2 +\\
  \Ah{\taushc, \tausht} + 
  \Ah{ \tausht, \Theta_h \taushc} + 
  \Ah{ \tausht, \tausht}.
  \end{multline} 
  By the Cauchy-Schwarz inequality, \begin{multline*}
   \Ah{ \tausht, \tausht}= \rho^{-1}\norm{\bdiv_h \tilde \btau_h}_{0,\O}^2 + 
   \texttt{a}_S \norm{h_\cF^{-1/2} \jump{ \btau_h}}_{0,\cF^*_h}^2 + 
   \int_{\O} \cC^{-1}\tilde \btau_h:\tilde \btau_h \\
   - 2 \int_{\cF^*_h} \mean{\rho^{-1}\bdiv_h \tilde \btau_h}\cdot \jump{\tilde \btau_h}
   \,\, \ge\,\, \texttt{a}_S \norm{h_\cF^{-1/2} \jump{ \btau_h}}_{0,\cF^*_h}^2 \\ %\textcolor{red}{C_0 \norm{ \tilde \btau_h }^2_{0,\O}}\\
     -2 \rho^{-1}\norm{h_\cF^{1/2}  \mean{\bdiv_h \tilde \btau_h}}_{0,\cF^*_h} \norm{h_\cF^{-1/2} \jump{ \btau_h}}_{0,\cF^*_h}
  \end{multline*}
and we deduce from \eqref{discTrace} and \eqref{L2Ph} that
\[
 \Ah{ \tausht, \tausht} \geq    (\texttt{a}_S -C_1 ) \norm{h_\cF^{-1/2} \jump{ \btau_h}}_{0,\cF^*_h}^2 ,
\]
with a constant $C_1$ independent of $h$ and $\lambda$. 

We proceed similarly for the terms in the right-hand side of \eqref{split}.
Indeed, it is straightforward  that 
\begin{multline*}
 \Ah{\taushc, \tausht} \geq -\rho^{-1}\norm{\bdiv \btau_h^c}_{0,\Omega}\norm{\bdiv_h \tilde \btau_h}_{0,\Omega} - C_2  
 \norm{\tilde \btau_h}_{0,\Omega}(\norm{\btau_h^c}_{0,\Omega}+ \norm{\bs_h}_{0,\Omega})- \\
 \rho^{-1}\norm{h_\cF^{1/2} \mean{\bdiv \btau^c_h}}_{0,\cF^*_h}\norm{h_\cF^{-1/2} \jump{ \btau_h}}_{0,\cF^*_h},
\end{multline*}
and using again \eqref{discTrace} and  \eqref{L2Ph}   we obtain 
\begin{multline*}
 \Ah{\taushc, \tausht} \geq 
 -C_3 \norm{h_\cF^{-1/2} \jump{ \btau_h}}_{0,\cF^*_h} \norm{\taushc} \geq \\
 - \frac{\alpha_A^c}{4}  \norm{\taushc}^2 -  C_4\norm{h_\cF^{-1/2} \jump{ \btau_h}}_{0,\cF^*_h}^2
\end{multline*}
with $C_4>0$ independent of $h$ and $\lambda$. Similar estimates  lead to  
\begin{multline*}
 \Ah{ \tausht, \Theta_h \taushc} \geq 
  -C_5  \norm{h_\cF^{-1/2} \jump{ \btau_h}}_{0,\cF^*_h} \norm{\Theta_h \taushc} 
  \geq \\
  -C_5 \norm{h_\cF^{-1/2} \jump{ \btau_h}}_{0,\cF^*_h} \norm{ \taushc}, 
\end{multline*}
where the last inequality follows from  \eqref{cota1}. We conclude that there exists 
$C_6>0$ independent of $h$ and $\lambda$ such that 
\begin{equation*}
 \Ah{ \tausht, \Theta_h \taushc} \geq 
  - \frac{\alpha_D^c}{4}  \norm{\taushc}^2 -  C_6 \norm{h_\cF^{-1/2} \jump{ \btau_h}}_{0,\cF^*_h}^2.
\end{equation*}
We then have shown that,
\begin{equation*}
 \Ah{\taush, \Theta_h \taushc+ \tausht}\ge 
 \frac{\alpha_A^c  }{2} \norm{  \taushc}^2 + \\\big(\texttt{a}_S  -  
 C_7 \big) 
 \norm{h_\cF^{-1/2}\jump{ \btau_h}}_{0,\cF_h}^2,
\end{equation*}
with $C_7 := C_1+ C_4+C_6$. Consequently,  
if $\texttt{a}_S > \texttt{a}_S^*:=C_7 + \frac{\alpha_A^c}{2}$,  
\begin{equation*}
 \Ah{\taush, \Theta_h \taushc+ \tausht}\ge \frac{\alpha_A^c  }{2} \Big( \norm{\taushc}^2 + 
 \norm{h_{\cF}^{-1/2} \jump{\btau}}^2_{0,\cF^*_h} \Big),
 %\geq\\ {\underbar C}^2 \frac{\alpha_D^c  }{2}\delta_0 \norm{ \taush }^2_{DG},
\end{equation*}
and thanks to \eqref{equivN},   we conclude that there exists $\alpha_{DG}>0$ such that, 
\begin{equation*}
 \Ah{\taush, \Theta_h \taushc+ \tausht}
\geq \alpha_{DG} 
\norm{\taush }_{DG} 
\Big(\norm{\Theta_h \taushc + \tausht}_{DG}\Big), 
\end{equation*}     
which gives \eqref{infsupABh}.
\end{proof}
In the sequel, we assume that the stabilization parameter is big enough $\texttt{a}_S > \texttt{a}_S^*$ so that the inf-sup condition \eqref{infsupABh} is guaranteed. The first consequence of this inf-sup condition is that the operator $\bT_h: \L^2(\O)^{\nxn}\times \L^2(\O)^{\nxn} \to \bcW_h\times \bcQ_h$ 
characterized, for any $(\bF, \bg) \in [\L^2(\O)^{\nxn}]^2$, by 
\begin{equation}\label{charcTDG}
 \Ah{\bT_h(\bF, \bg), (\btau_h,\bs_h)} = \B{(\bF, \bg), (\btau_h,\bs_h)} \quad \forall (\btau_h,\bs_h)\in \bcW_h\times \bcQ_h
\end{equation}
is well-defined, symmetric with respect to $A_h(\cdot, \cdot)$ and there exists a constant $C>0$ independent of 
$\lambda$ and $h$ such that 
\begin{equation}\label{bTh}
\norm{\bT_h (\bF, \bg)}_{DG} \leq C \norm{(\bF, \bg)}_{0,\O}\quad \forall (\bF, \bg) \in [\L^2(\O)^{\nxn}]^2.
\end{equation}

We observe that if $(\kappa_h,\sigmarh)\in\mathbb{R}\times\bcW_h\times\bcQ$ is a solution of problem \eqref{DGshort} if and only if $(\mu_h,\sigmarh)$, with $\mu_h=1/(1+\kappa_h)$ is an eigenpair  of $\bT_h$, i.e.
$$
\bT_h\sigmarh=\frac{1}{1+\kappa_h}\sigmarh.
$$

Analogously to the continuous case, we prove that the discrete resolvent associated to the discrete operator $\bT_h$ is bounded.
\begin{theorem}\label{cea}
Assume that $(\tilde \bsig, \tilde \br):=\bT(\bF, \bg)\in \H^t(\bdiv, \O)\times \H^t(\O)^{\nxn}$ for some $t>1/2$.
Then, 
\begin{equation}\label{Cea}
\norm{(\bT - \bT_h)(\bF, \bg)}_{DG} \leq \left(1 + \dfrac{M_{DG}}{\alpha_{DG}}\right) 
\inf_{\taush\in \bcW_h\times \bcQ_h} \norm{\bT(\bF, \bg) - \taush}^*_{DG}.
\end{equation}
%\Big( 
% \norm{(\tilde \bsig -  \Pi_h \tilde\bsig , \tilde \br - \mathcal{S}_h \tilde\br) }_{DG} + 
% \norm{h_{\cF}^{1/2} \mean{\bdiv (\tilde\bsig-\Pi_h\tilde\bsig)}}_{0,\cF^*_h}
% \Big)
Moreover, the error estimate  
\begin{equation}\label{asymp}
 \norm{(\bT - \bT_h)(\bF, \bg)}_{DG} \leq \, C\,  h^{\min(t, k)}\,  
 \Big( \norm{\tilde \bsig}_{\H^t(\bdiv, \O)} +  \norm{\tilde \br}_{\H^t(\O)^{\nxn}} \Big), 
\end{equation}
holds true with a constant  $C>0$ independent of $h$ and $\lambda$.
\end{theorem}
\begin{proof}
We first notice that the DG approximation \eqref{charcTDG} is consistent with regards to its continuous counterpart \eqref{charcT} in the sense that 
\begin{equation}\label{consistency}
\Ah{(\bT-\bT_h)(\bF, \bg), \taush} = 0 \quad \forall \taush\in \bcW_h\times \bcQ_h.
\end{equation}
Indeed, by definition, 
 \begin{multline}
 \label{id1}
  \Ah{(\tilde \bsig, \tilde \br), \taush} = \int_{\O} \rho^{-1}\bdiv\tilde\bsig \cdot \bdiv_h \btau_h + \B{(\tilde \bsig, \tilde \br), \taush}\\ -
\int_{\cF^*_h}  \mean{\rho^{-1}\bdiv\tilde\bsig} \cdot \jump{\btau_h}.
 \end{multline}
 It is straightforward to deduce from \eqref{charcT} 
 %\textcolor{red}{the correct equation must be \eqref{charcT}} that 
 \begin{equation}\label{you}
 \nabla \left(\rho^{-1} \bdiv \tilde \bsig\right) = \cC^{-1}(\tilde \bsig - \bF) + \tilde \br - \bg\quad 
 \text{and} \quad 
 (\tilde \bsig - \tilde \bsig^\t)/2 = (\bF - \bF^\t)/2.
 \end{equation}
Moreover, an  integration by parts  yields
\begin{multline*}
 \int_{\O} \rho^{-1}\bdiv\tilde \bsig \cdot \bdiv_h \btau_h  = 
 -\sum_{K\in \cT_h} \int_K \nabla(\rho^{-1}\bdiv \tilde \bsig): \btau_h + \sum_{K\in \cT_h} \int_{\partial K} \rho^{-1}\bdiv
 \tilde \bsig \cdot \btau_h\bn_K\\
 = -\sum_{K\in \cT_h} \int_K \nabla(\rho^{-1}\bdiv \tilde \bsig): \btau_h + \int_{\cF^*_h} \mean{\rho^{-1}\bdiv
     \tilde \bsig}\cdot \jump{\btau_h}.
\end{multline*}
Substituting back the last identity and \eqref{you} into  \eqref{id1} 
we obtain 
\begin{equation*}
 \Ah{(\tilde \bsig, \tilde \br), \taush} = \B{(\bF, \bg), \taush}  \quad \forall \taush \in \bcW_h\times \bcQ_h
\end{equation*}
and \eqref{consistency} follows.

The C\'ea estimate \eqref{Cea} follows now in the usual way by taking advantage of \eqref{consistency}, the inf-sup condition \eqref{infsupABh}, estimate  \eqref{boundA2}, and the triangle inequality.

It follows from \eqref{Cea} that  
\begin{equation}\label{yaesta}
\norm{(\bT - \bT_h)(\bF, \bg)}_{DG} \leq \left(1 + \dfrac{M_{DG}}{\alpha_{DG}}\right) 
 \norm{(\tilde \bsig, \tilde \br) - (\Pi_h\tilde \bsig, \mathcal S_h\tilde \br)}^*_{DG}.
\end{equation}
Using the interpolation error estimates \eqref{asymp0}, \eqref{asympDiv} and \eqref{asymQ} we immediately obtain 
\begin{equation}\label{cotaA}
\norm{(\tilde \bsig, \tilde \br) - (\Pi_h\tilde \bsig, \mathcal S_h\tilde \br)}_{DG} = \norm{(\tilde \bsig, \tilde \br) - (\Pi_h\tilde \bsig, \mathcal S_h\tilde \br)}\leq 
C_0\,  h^{\min(t, k)}\,  
 \Big( \norm{\tilde \bsig}_{\H^t(\bdiv, \O)} +  \norm{\tilde \br}_{\H^t(\O)^{\nxn}} \Big).
\end{equation}
Moreover, we notice that 
\begin{equation*}
  \norm{h_{\cF}^{1/2} \mean{\bdiv (\tilde\bsig-\Pi_h\tilde\bsig)}}_{0,\cF^*_h}  
  \leq   
   \sum_{K\in \cT_h} \sum_{F\in \cF(K)} h_F\norm{ \bdiv (\tilde\bsig-\Pi_h\tilde\bsig) }^2_{0,F}.
\end{equation*}
Under the regularity hypotheses on  $\tilde\bsig$, the commuting diagram property satisfied by $\Pi_h$, the trace theorem and standard scaling arguments give 
\[
 h_F^{1/2}\norm{ \bdiv (\tilde\bsig-\Pi_h\tilde\bsig) }_{0,F} = h_F^{1/2}\norm{ \bdiv \tilde\bsig- \mathcal R_K \bdiv\tilde\bsig }_{0,F} 
 \leq C_2 h_K^{\min(t,k)} \norm{\bdiv \tilde\bsig}_{t,K}
\]
for all $F\in \cF(K)$, where the $\L^2(K)$-orthogonal projection $\mathcal R_K:= \mathcal R_h|_K$ 
onto $\cP_{k-1}(K)$ is applied componentwise. It follows that 
\begin{equation}\label{newE}
	\norm{h_{\cF}^{1/2} \mean{\bdiv (\tilde\bsig-\Pi_h\tilde\bsig)}}_{0,\cF^*_h} 
	\leq C_3 h_K^{\min(t,k)} \left( \sum_{K\in \cT_h} \norm{\bdiv \tilde\bsig}_{t,K}^2 \right)^{1/2} \leq 
	C_3 h_K^{\min(t,k)} \norm{\bdiv \tilde\bsig}_{t,\O}.
\end{equation}
Combining  \eqref{newE} and \eqref{cotaA} with \eqref{yaesta} proves the asymptotic error estimate \eqref{asymp}.
\end{proof}

\begin{lemma}\label{final}
For all $s\in (0,\ws)$, there exists a constant  $C>0$ independent of $h$ and $\lambda$, such that for all $\sigmar\in \bcW\times \bcQ$
$$
 \norm{(\bT - \bT_h)\bP\sigmar}_{DG} \leq \, C\,  h^s\,  
 \norm{\bdiv \bsig}_{0,\Omega}.
$$
\end{lemma}
\begin{proof}
The result is a consequence of Theorem \ref{cea} by noticing that, by virtue of Lemma \ref{reg} and Assumption~\ref{assumpt1}, $\bT\circ\bP\subset\{(\btau,\br)\in[\H^s(\O)^{n\times n}]^2:\,\bdiv\btau\in\H^1(\O)^n\}$ for all $s\in(0,\ws)$.
\end{proof}

\begin{lemma}\label{TmenosTh}
For all $s\in (0,\ws)$, there exists a constant $C>0$ independent of $h$ and $\lambda$ such that 
\[
\norm{(\bT - \bT_h)\taush}_{DG} \leq C \, h^s \, \norm{\taush}_{DG}\quad \forall \taush \in \bcW_h\times \bcQ_h.
\]
\end{lemma}
\begin{proof}
For any $\btau_h\in \bcW_h$ we consider the splitting  $\btau_h = \btau_h^c + \tilde \btau_h$ with 
$\btau_h^c:=\mathcal{I}_h \btau_h\in \bcW_h^c$. We have that 
\begin{multline*}
(\bT - \bT_h)\taush = (\bT - \bT_h)(\tilde \btau_h,\0) + (\bT - \bT_h)(\btau_h^c, \bs_h)\\ = (\bT - \bT_h)(\tilde \btau_h,\0) +
(\bT - \bT_h)\bP_h(\btau_h^c, \bs_h),
\end{multline*}
where the last identity is due to the fact that $(\bI - \bP_h)(\btau_h^c, \bs_h)\in \bcK_h\times \bcQ_h$ 
and  $\bT - \bT_h$ vanishes identically on this subspace. It follows that 
\begin{multline*}
(\bT - \bT_h)\taush = (\bT - \bT_h)(\tilde \btau_h,\0) +
(\bT - \bT_h)(\bP_h - \bP)(\btau_h^c, \bs_h) + (\bT - \bT_h)\bP(\btau_h^c, \bs_h),
\end{multline*}
and the triangle inequality together with \eqref{bT} and  \eqref{bTh} yield
\begin{multline*}
\norm{(\bT - \bT_h)\taush}_{DG} \leq  \norm{(\bT - \bT_h)(\tilde \btau_h,\0)}_{DG} +
\norm{(\bT - \bT_h)(\bP_h - \bP)(\btau_h^c, \bs_h)}_{DG}\\ + \norm{(\bT - \bT_h)\bP(\btau_h^c, \bs_h)}_{DG} 
\leq 
\Big(\norm{\bT}_{\mathcal{L}([\L^2(\O)^{\nxn}]^2, \bcW\times \bcQ)} + 
\norm{\bT_h}_{\mathcal{L}([\L^2(\O)^{\nxn}]^2, \bcW_h\times \bcQ_h)} \Big)\\
 \Big( \norm{\tilde \btau_h}_{0,\O}  + \norm{(\bP_h - \bP)(\btau_h^c, \bs_h)}\Big) + 
 \norm{(\bT - \bT_h)\bP(\btau_h^c, \bs_h)}_{DG}.
\end{multline*}
Using \eqref{L2Ph}, Lemma \ref{Ph}, Assumption \ref{assumpt1} and Lemma \ref{final} we have that
\[
\norm{\tilde \btau_h}_{0,\O} \leq C h \norm{\btau_h}_{\bcW(h)},
\]
\[
\norm{(\bP_h - \bP)(\btau_h^c, \bs_h)} \leq C h^s \norm{\bdiv \btau_h^c}_{0,\O} \leq C h^s \norm{\btau_h}_{\bcW(h)}
\]
and 
\[
\norm{(\bT - \bT_h)\bP(\btau_h^c, \bs_h)}_{DG} \leq C h^s \norm{\bdiv \btau_h^c}_{0,\O} \leq C h^s \norm{\btau_h}_{\bcW(h)}
\]
respectively, which gives the result.
\end{proof}

\section{Spectral correctness of the DG method}
\label{APPROX}

The convergence analysis follows the same steps introduced in \cite{DNR1,DNR2}, we only need to adapt 
it to the DG context, cf. also \cite{BuffaPerugia}.  

For the sake of brevity, we
will denote in  this section $\bcX:=\bcW\times\bcQ$, $\bcX_h:=\bcW_h\times\bcQ_h$
and $\bcX(h) := \bcW(h)\times \bcQ$. Moreover, when no confusion can arise, we
will use indistinctly $\bx$, $\by$, etc. to denote elements in $\bcX$
and, analogously, $\bx_h$, $\by_h$, etc. for those in $\bcX_h$.
Finally, we will use $\norm{\cdot}_{\mathcal{L}(\bcX_h,  \bcX(h))}$
to denote the norm of an operator restricted to the discrete
subspace $\bcX_h$; namely, if $\bS:\bcX(h)\to \bcX(h)$, then
\begin{equation}\label{norm}
\norm{\bS}_{\mathcal{L}(\bcX_h, \bcX(h))}:=\sup_{\0\neq\bx_h\in\bcX_h}
\frac{\norm{\bS\bx_h}_{DG}}{\norm{\bx_h}_{DG}}.
\end{equation}

\begin{lemma}\label{ThDG0}
If $z \in\mathbb{D}\setminus \sp(\bT)$, there exists $h_0>0$ such that if $h\leq h_0$,
\[
\norm{(z \bI - \bT_h) \bx_h }_{DG} \geq C\dist\big(z,\sp(\bT)\big)|z|\, \norm{\bx_h}_{DG} \quad \forall \bx_h \in \bcX_h.
\]
with $C>0$ independent of $h$ and $\lambda$.
\end{lemma}
\begin{proof}
It follows from 
\[
(z \bI - \bT_h) \bx_h = (z \bI - \bT) \bx_h + (\bT - \bT_h ) \bx_h
\]
and Lemma \ref{TDG} that 
\[
\norm{(z \bI - \bT_h) \bx_h}_{DG} \geq \Big(C\dist\big(z,\sp(\bT)\big)|z| - \norm{\bT - \bT_h}_{\mathcal{L}(\bcX_h, \bcX(h))}\Big) \norm{\bx_h}_{DG}
\]
and the result follows from Lemma \ref{TmenosTh}.
\end{proof}
\begin{lemma} \label{ThDG}
If $z \in\mathbb{D}\setminus \sp(\bT)$, there exists $h_0>0$ such that if $h\leq h_0$,
\[
\norm{(z \bI - \bT_h) \bx }_{DG} \geq C \dist\big(z,\sp(\bT)\big)|z|^2\, \norm{\bx}_{DG} \quad \forall \bx \in \bcX(h),
\]
with $C>0$ independent of $h$ and $\lambda$.
\end{lemma}

\begin{proof}
Given $\bx\in \bcX(h)$ we let 
\[
\bx_h^* = \bT_h \bx \in \bcX_h.
\]
We deduce from the identity 
\[
(z \bI - \bT_h) \bx_h^* = \bT_h (z\bI - \bT_h)\bx
\]
and from Lemma \ref{ThDG0} that
\[
C\dist\big(z,\sp(\bT)\big)|z| \norm{ \bx_h^* }_{DG} \leq \norm{(z \bI - \bT_h) \bx_h^*}_{DG} \leq 
\norm{\bT_h}_{	\mathcal{L}(\bcX(h), \bcX_h)} \norm{(z \bI - \bT_h)\bx}_{DG}.
\]
This and the triangle inequality leads to
\begin{multline*}
\norm{\bx}_{DG} \leq |z|^{-1} \norm{\bx^*_h}_{DG} + |z|^{-1} \norm{ (z \bI - \bT_h)\bx }_{DG}
\\ 
\leq|z|^{-1} \left( 1 + \frac{\norm{\bT_h}_{\mathcal{L}(\bcX(h), \bcX_h)}}{C\dist\big(z,\sp(\bT)\big)|z|} \right)\norm{ (z \bI - \bT_h)\bx }_{DG}.
\\
\leq|z|^{-1} \left( \frac{C\dist\big(z,\sp(\bT)\big)|z|+\norm{\bT_h}_{\mathcal{L}(\bcX(h), \bcX_h)}}{C\dist\big(z,\sp(\bT)\big)|z|} \right)\norm{ (z \bI - \bT_h)\bx }_{DG}.
\end{multline*}
Hence,
\begin{equation*}
C|z|\left(\frac{C\dist\big(z,\sp(\bT)\big)|z|}{\norm{\bT_h}_{\mathcal{L}(\bcX(h), \bcX_h)}+C\dist\big(z,\sp(\bT)\big)|z|}\right)\norm{\bx}_{DG}\leq \norm{(z \bI - \bT_h) \bx }_{DG}.
\end{equation*}
Now, using that $\dist\big(z,\sp(\bT)\big)\leq |z|\leq 1$  and $\|\bT_h\|_{\mathcal{L}(\bcX(h), \bcX_h)}\leq C'$ (with $C'$ independent of $\lambda$), from the estimate above we derive
\begin{equation*}
C|z|^2 \dist\big(z,\sp(\bT)\big)\norm{\bx}_{DG} \leq \norm{(z \bI - \bT) \taus }_{DG},
\end{equation*}
and the result follows. 
\end{proof}

\begin{remark}\label{rem2}
If $E$ is a compact subset of $\mathbb{D} \setminus\sp(\bT)$ and $h$ is small enough, we deduce from Lemma \ref{ThDG} that $(z\bI-\bT_h):\bcX(h)\rightarrow\bcX(h)$ is invertible for all $z\in E$. Hence, $E\subset\mathbb{D}\backslash\sp(\bT_h)$. Consequently, for $h$ small enough, the numerical method does not introduce spurious eigenvalues. Moreover, we have that there exists a constant $C>0$ independent of $h$ and $\lambda$ such that, for all $z\in E$,
\begin{equation*}\label{resid}
\norm{\big(z \bI - \bT_h \big)^{-1}}_{\mathcal{L}(\bcX(h),\bcX(h))} \leq \frac{C}{\dist(E,\sp(\bT))|z|^2}.
\end{equation*}
\end{remark}

%\begin{remark}
%To prove Lemma \ref{cont_resolvent} and Lemma \ref{disc_resolvent}, we have consider $\lambda$ large. However, it is possible to prove  same results considering a fixed $\lambda$. 
%\end{remark}

For
$\bx\in\bcX(h)$ and $\mathbb E$ and $\mathbb F$ closed subspaces of $\bcX(h)$, we set
$\delta(\bx,\mathbb E):=\inf_{\by\in\mathbb E}\norm{\bx-\by}_{DG}$,
$\delta(\mathbb E,\mathbb F):=\sup_{\by\in\mathbb E:\,\norm{\by}=1}\delta(\by,\mathbb F)$,
and $\gap(\mathbb E,\mathbb F):=\max\set{\delta(\mathbb E,\mathbb F),\delta(\mathbb F,\mathbb E)}$,
the latter being the so called \textit{gap} between subspaces $\mathbb E$ and
$\mathbb F$.

Given an isolated eigenvalue $\kappa\neq 1$  of $\bT$, we define 
$$\texttt{d}_{\kappa}:=\frac{1}{2}\dist\big(\kappa,\sp(\bT)\setminus\{\kappa\}\big).$$
 It follows that  the closed disk $D_\kappa:=\{z\in\mathbb{C}:\quad |z-\kappa|\leq \dd_\kappa\}$ of the complex plane, with center $\kappa$ and boundary $\gamma$ is such that   
$D_\kappa \cap \sp(\bT) = \set{\kappa}$.  We deduce from Remark \ref{roof} that the operator 
$\bcE:=\frac{1}{2\pi i}
\int_{\gamma}\left(z\bI-\bT\right)^{-1}\, dz:\bcX(h)\longrightarrow \bcX(h)$ is well-defined and bounded uniformly in 
$h$. Moreover, $\bcE|_{\bcX}$ is a spectral projection in $\bcX$ onto the (finite dimensional)  eigenspace $\bcE(\bcX)$ corresponding to the eigenvalue $\kappa$ of $\bT$. In fact, 
\begin{equation}\label{equa}
\bcE(\bcX(h)) = \bcE(\bcX).
\end{equation}
To prove this, let $\kappa^*\in D_\kappa$ be an eigenvalue of $\bT:\, \bcX(h)\to \bcX(h)$ and 
$\bx^*\in \bcX(h)$ be the corresponding eigenfunction. Since $\kappa^*\neq 0$ and $\bT(\bcX(h))\subset \bcX$, we actually have that $\bx^*\in \bcX$. Then, necessarily,  $\kappa^*=\kappa$ and taking into account that $\bcE(\bcX)$ is the eigenspace associated with $\kappa$ we deduce \eqref{equa}. 

Similarly, we deduce from Remark \ref{rem2} that, for $h$ small enough, the operator 
$\bcE_h:=\frac{1}{2\pi i}
\int_{\gamma}\left(z\bI-\bT_h\right)^{-1}\, dz:\bcX(h)\longrightarrow \bcX(h)$ is also well-defined and bounded uniformly in $h$. Moreover, $\bcE_h|_{\bcX_h}$ is a projector in $\bcX_h$ onto the eigenspace $\bcE_h(\bcX_h)$  corresponding to the eigenvalues of $\bT_h:\, \bcX_h \to \bcX_h$ contained in $\gamma$. The same arguments as above show that we also have, 
\begin{equation*}\label{equah}
\bcE_h(\bcX(h)) = \bcE_h(\bcX_h).
\end{equation*}
Our aim now is to compare $\bcE_h(\bcX_h)$ to $\bcE(\bcX)$ in terms of the gap $\gap$. In order to do that,  we assume the following regularity assumption $\bcE(\bcX)\subset\H^t(\bdiv,\O)\times\H^t(\O)^{n\times n}$ with $t>s$. 

\begin{lemma}\label{lot}
There exists $C>0$, independent of $h$ and $\lambda$, such that
\begin{equation}\label{E-Eh}
\displaystyle \norm{\bcE - \bcE_h}_{\mathcal{L}(\bcX_h, \bcX(h))} \leq \frac{C}{\texttt{d}_{\kappa}} \norm{\bT - \bT_h}_{\mathcal{L}(\bcX_h, \bcX(h))}.
\end{equation}
\end{lemma}
\begin{proof}
We deduce from the identity 
\begin{equation*}\label{identRes}
\left(z\bI-\bT\right)^{-1}-\left(z\bI-\bT_h\right)^{-1} = \left(z\bI-\bT\right)^{-1}(\bT - \bT_h)\left(z\bI-\bT_h\right)^{-1}
\end{equation*}
that, for any $\bx_h\in \bcX_h$,

\begin{multline*}
\norm{(\bcE - \bcE_h)\bx_h}_{DG}\leq \frac{1}{2\pi}\int_{\gamma}\norm{[\left(z\bI-\bT\right)^{-1}-\left(z\bI-\bT_h\right)^{-1}]\bx_h}_{DG}|dz|\\
=\frac{1}{2\pi}\int_{\gamma}\norm{[\left(z\bI-\bT\right)^{-1}(\bT - \bT_h)\left(z\bI-\bT_h\right)^{-1}]\bx_h}_{DG}|dz|\\
\leq \frac{1}{2\pi}\int_{\gamma}\norm{\left(z\bI-\bT\right)^{-1}}_{\mathcal L(\bcX(h), \bcX(h))}\norm{\bT - \bT_h}_{\mathcal L(\bcX_h, \bcX(h))}\norm{\left(z\bI-\bT_h\right)^{-1}}_{\mathcal L(\bcX_h, \bcX_h)}\norm{\bx_h}_{DG}|dz|
\end{multline*}
and the result follows from Lemmas \ref{TDG} and \ref{ThDG}, the definition \eqref{norm} and the fact that for all $z\in\gamma$, $|z|\geq\kappa-\dd_{\kappa}\geq\frac{1}{2}\kappa.$
\end{proof}

\begin{theorem}
\label{conv}
There exists a constant $C>0$ independent of $h$ and $\lambda$ such that
\[
\gap(\bcE(\bcX), \bcE_h(\bcX_h)) \leq C \Big( \frac{\norm{\bT - \bT_h}_{\mathcal{L}(\bcX_h, \bcX(h))}}{{\dd_{\kappa}}} +
\delta(\bcE(\bcX), \bcX_h)\Big).
\]
\end{theorem}
\begin{proof}
As $\bcE_h$ is a projector,  for $h$ sufficiently small,  we have that $\bcE_h\bx_h=\bx_h$ for all $\bx_h\in\bcE_h(\bcX_h)$.  It follows from \eqref{equa} that $\bcE\bx_h \in \bcE(\bcX)$, which leads to
\[
\delta(\bx_h,\bcE(\bcX))
\leq\norm{\bcE_h\bx_h-\bcE\bx_h}_{DG}
\leq\norm{\bcE_h-\bcE}_{\mathcal{L}(\bcX_h, \bcX(h))}\norm{\bx_h}_{DG}
\]
for all $\bx_h\in\bcE_h(\bcX_h)$. We deduce from \eqref{E-Eh}  that 
\begin{equation}\label{cE-cEh}
\delta( \bcE_h(\bcX_h), \bcE(\bcX))\leq \frac{C}{\dd_{\kappa}} \norm{\bT - \bT_h}_{\mathcal{L}(\bcX_h, \bcX(h))}.
\end{equation}

On the other hand,  as $\bcE\bx=\bx$
for all $\bx\in\bcE(\bcX)$, for $h$ small enough and $\by_h\in \bcX_h$,
\begin{multline*}
\norm{\bx-\bcE_h\by_h}_{DG}
 \leq\norm{\bcE(\bx-\by_h)}_{DG}
+\norm{(\bcE-\bcE_h)\by_h}_{DG}\leq \\
\norm{\bcE}_{\mathcal{L}( \bcX(h), \bcX(h))} \norm{(\bx-\by_h)}_{DG}
+\norm{(\bcE-\bcE_h)}_{\mathcal{L}(\bcX_h, \bcX(h))} \norm{\by_h}_{DG}
\\
 \leq \big(\norm{\bcE_h}_{\mathcal{L}(\bcX(h), \bcX(h))}+2\norm{\bcE}_{\mathcal{L}(\bcX(h), \bcX(h))}\big)\norm{\bx - \by_h}_{DG}
+\norm{\bcE-\bcE_h}_{\mathcal{L}(\bcX_h, \bcX(h))}\norm{\bx}_{DG}.
\end{multline*}
Consequently,
\[
\delta(\bx, \bcE_h(\bcX_h)) \leq C (\delta(\bx, \bcX_h) + \norm{\bcE-\bcE_h}_{\mathcal{L}(\bcX_h, \bcX(h))} )
\]
for all $\bx \in \bcE(\bcX)$ such that $\norm{\bx}_{DG} = 1$ and using that the eigenspace $\bcE(\bcX)$ is finite dimensional we deduce that
\[
\delta(\bcE(\bcX), \bcE_h(\bcX_h)) \leq C (\delta(\bcE(\bcX), \bcX_h) + \norm{\bcE-\bcE_h}_{\mathcal{L}(\bcX_h, \bcX(h))} )
\]
and the result follows from the last estimate and \eqref{cE-cEh}.
\end{proof}
\begin{theorem}
Let $\kappa\neq 1$ be an eigenvalue of $\bT$ of algebraic multiplicity $m$  and 
let $D_\kappa$ be a closed disk in the complex plane centered at $\kappa$ with  boundary $\gamma$ such that 
$D_\kappa \cap \sp(\bT) = \set{\kappa}$. Let $\kappa_{1, h}, \ldots, \kappa_{m(h), h}$ be the eigenvalues of 
$\bT_h:\, \bcX_h \to \bcX_h$  lying in  $D_\kappa$ and repeated according to their algebraic multiplicity. 
Then, we have that $m(h) = m$ for $h$ sufficiently small and 
\[
\lim_{h\to 0} \max_{1\leq i \leq m} |\kappa -  \kappa_{i, h}| =0.
\]
Moreover, if $\bcE(\bcX)$ is the eigenspace corresponding to $\kappa$ and $\bcE_h(\bcX_h)$ is the
$\bT_h$-invariant subspace of $\bcX_h$ spanned by the eigenspaces  corresponding to
$\set{\kappa_{i, h},\hspace{0.1cm} i = 1,\ldots, m}$ then
\[
\lim_{h \to 0} \gap(\bcE(\bcX), \bcE_h(\bcX_h)) = 0.
\]
\end{theorem}
\begin{proof}
We deduce from Lemma \ref{TmenosTh} that 
\[
\lim_{h\to 0}\norm{\bT - \bT_h}_{\mathcal{L}(\bcX_h, \bcX(h))} = 0.
\] 
Moreover, as $\bcE(\bcX) \subset \H^t(\bdiv, \O)\times \H^t(\O)^{\nxn}$, it follows from \eqref{asymp} that

%\eqref{asympDiv} and \eqref{asymQ} that 
%\textcolor{red}{we think that it is enough to apply \eqref{asymp} to conclude}
\[
\lim_{h\to 0}\delta(\bcE(\bcX), \bcX_h) = 0.
\]
Hence, by virtue of Theorem \ref{conv}, we have that 
\[
\lim_{h\to 0} \gap(\bcE(\bcX), \bcE_h(\bcX_h)) = 0,
\]
and, as a consequence, $\bcE(\bcX)$ and $\bcE_h(\bcX_h)$ have the same dimension provided $h$ is sufficiently small. Finally, being $\kappa$ an isolated eigenvalue and the radius of the circle $\gamma$ arbitrary,  we deduce that 
\[
\lim_{h\to 0} \max_{1\leq i \leq m} |\kappa -  \kappa_{i, h}| =0.
\]
\end{proof}
\section{Asymptotic error estimates}
\label{ASYMP}

Along this section we fix a particular eigenvalue $\kappa\neq 1$ of $\bT$. We wish to obtain error 
estimates for the eigenfunctions and the eigenvalues in terms of the quantity
\[
\delta^* (\bcE(\bcX) , \bcX_h):= \sup_{\bx\in \bcE(\bcX), \norm{\bx}=1}\inf_{\bx_h\in \bcX_h} \norm{\bx - \bx_h}^*_{DG}.
\] 
%We introduce the parameter
%\[
%\xi_h := \delta^* (\bcE(\bcX) , \bcX_h).
%\]
%We notice that 
%\[
%\bcE(\bcX(h)) = \bcE(\bcX).
%\]
\begin{theorem}\label{hatgap}
For $h$ small enough, there exists a constant $C$ independent of $h$ such that 
\begin{equation}\label{aste}
\gap\big(\bcE(\bcX), \bcE_h(\bcX_h) \big)\leq \frac{C}{\dd_{\kappa}} \delta^* (\bcE(\bcX) , \bcX_h).
\end{equation}
 \end{theorem}
\begin{proof}
As $\bcE(\bcX(h))=\bcE(\bcX)$ and $\bcE_h(\bcX(h))=\bcE_h(\bcX_h)$, it is equivalent to show that 
\begin{equation*}\label{gogo}
\gap\Big(\bcE(\bcX(h)), \bcE_h(\bcX(h))\Big)\leq \frac{C}{\dd_{\kappa}} \delta^* (\bcE(\bcX) , \bcX_h).
\end{equation*}
We consider here again the disk $D_{\kappa}$ centered at $\kappa$ with radius $\dd_{\kappa}$ and boundary $\gamma$. We first notice that for all $z\in\gamma$ 
\[
\left(z\bI-\bT\right)^{-1}-\left(z\bI-\bT_h\right)^{-1} = \left(z\bI-\bT_h\right)^{-1}(\bT - \bT_h)\left(z\bI-\bT\right)^{-1},
\]
which implies
\begin{multline}\label{cuqui}
\norm{ (\bcE - \bcE_h)|_{\bcE(\bcX)} } \leq \frac{1}{2\pi }\int_{\gamma}\norm{\left(z\bI-\bT\right)^{-1}-\left(z\bI-\bT_h\right)^{-1}|_{\bcE(\bcX)}}|dz|\\
=\frac{1}{2\pi}\int_{\gamma}\norm{\left(z\bI-\bT_h\right)^{-1}(\bT - \bT_h)\left(z\bI-\bT\right)^{-1}|_{\bcE(\bcX)}}|dz|\\
\leq \frac{1}{2\pi}\int_{\gamma} \norm{\left(z\bI-\bT_h\right)^{-1}}_{\mathcal{L}(\bcX(h), \bcX(h))}\norm{(\bT - \bT_h)|_{\bcE(\bcX)}}_{\mathcal{L}(\bcX, \bcX(h))}\norm{\left(z\bI-\bT\right)^{-1}}_{\mathcal{L}(\bcX, \bcX(h))}|dz|\\
\leq\frac{C}{\dd_{\kappa}}\norm{(\bT - \bT_h)|_{\bcE(\bcX)}}_{\mathcal{L}(\bcX, \bcX(h))}
\end{multline}
%\begin{multline}\label{cuqui}
%\norm{ (\bcE - \bcE_h)|_{\bcE(\bcX)} } \leq C \norm{\left(z\bI-\bT_h\right)^{-1}}_{\mathcal{L}(\bcX(h), \bcX(h))}\\ \norm{(\bT - \bT_h) \left(z\bI-\bT\right)^{-1}|_{\bcE(\bcX)} }_{\mathcal{L}(\bcX, \bcX(h))} 
%\leq C \norm{(\bT - \bT_h)|_{\bcE(\bcX)} }_{\mathcal{L}(\bcX, \bcX(h))},\\
%\leq \frac{C}{\dd^2}\norm{(\bT - \bT_h)|_{\bcE(\bcX)} }_{\mathcal{L}(\bcX, \bcX(h))}
%\end{multline}
Now, on the one hand, it is clear that
\[
\delta\Big(\bcE(\bcX(h)), \bcE_h(\bcX(h))\Big) \leq \norm{ (\bcE - \bcE_h)|_{\bcE(\bcX)} }_{\mathcal{L}(\bcX, \bcX(h))}.
\]
On the other hand,  \eqref{cuqui}, the C\'ea estimate given by Theorem \ref{cea} and the fact that  
$\bcE(\bcX)$ is finite dimensional yield 
\begin{equation}\label{yum}
\norm{ (\bcE - \bcE_h)|_{\bcE(\bcX)} }_{\mathcal{L}(\bcX, \bcX(h))} \leq \frac{C}{\dd_{\kappa}} \delta^* (\bcE(\bcX) , \bcX_h),
\end{equation}
which proves that 
\begin{equation}\label{gogo1}
\delta\Big(\bcE(\bcX(h)), \bcE_h(\bcX(h))\Big)   \leq \frac{C}{\dd_{\kappa}} \delta^* (\bcE(\bcX) , \bcX_h).
\end{equation}
Consequently, as $\bcE(\bcX)\subset \H^t(\bdiv, \O)\times \H^t(\O)^{\nxn}$, we have that 
\begin{equation}\label{Gamma}
\lim_{h\to 0} \delta\Big(\bcE(\bcX(h)) , \bcE_h(\bcX(h))\Big) = 0.
\end{equation}
It is shown in \cite{DNR2} that \eqref{Gamma} 
implies that, for $h$ small enough, $\Lambda_h:=\bcE_h|_{\bcE(\bcX)}: \bcE(\bcX) \to \bcE_h(\bcX(h))$ 
is bijective and $\Lambda_h^{-1}$ exists and is uniformly bounded with respect to 
$h$. Furthermore, it holds that,
\[
\sup_{\bx_h \in \bcE_h(\bcX(h)), \norm{\bx_h}_{DG} = 1} \norm{\Lambda_h^{-1} \bx - \bx}_{DG} \leq 
2 \sup_{\by \in \bcE(\bcX(h)), \norm{\by}_{DG} = 1} \norm{\Lambda_h \by - \by}_{DG}.
\]
Hence,
\[
\delta\Big(\bcE_h(\bcX(h)), \bcE(\bcX(h))\Big) \leq \sup_{\bx_h \in \bcE_h(\bcX(h)), \norm{\bx_h}_{DG} = 1} \norm{\bx_h - \Lambda_h^{-1} \bx}_{DG} \leq  2 \sup_{\by \in \bcE(\bcX), \norm{\by}_{DG} = 1} \norm{\bcE \by - \bcE_h \by}_{DG},
\]
and \eqref{yum} shows that we also have $\displaystyle\delta( \bcE_h(\bcX(h)), \bcE(\bcX(h)))\leq \frac{C}{\dd_{\kappa}} \delta^* (\bcE(\bcX) , \bcX_h)$, 
and the result follows from this last estimate and \eqref{gogo1}.
\end{proof}
%We notice that, as $\bcE_h(\bcX(h))=\bcE_h(\bcX)$, we also deduce from the last result that 
%\[
%\gap(\bcE_h(\bcX_h), \bcE(\bcX))\leq C \delta^* (\bcE(\bcX) , \bcX_h).
%\]
%\begin{theorem}
%There exists a constant $C$ independent of $h$ such that 
%\[
%\gap(\bcE_h(\bcX_h), \bcE(\bcX))\leq C \delta^* (\bcE(\bcX) , \bcX_h).
%\]
%\end{theorem}
%\begin{proof}
%Let $\bx\in \bcE(\bcX)$ with $\norm{\bx}=1$. We have that (we are using $\bcE_h(\bcX(h))=\bcE_h(\bcX) = \bcE_h(\bcX_h)$)
%\[
%\inf_{\by_h\in \bcE_h(\bcX_h)} \norm{\bx - \by_h} \leq \norm{\bcE\bx - \bcE_h\bx}\leq 
%\norm{(\bcE - \bcE_h)|_{\bcE(\bcX)}}_{\mathcal{L}(\bcX, \bcX(h))},
%\]
%which proves that $\delta(\bcE(\bcX), \bcE_h(\bcX_h))\leq C \delta^* (\bcE(\bcX) , \bcX_h)$.
%
%On the other hand,  
%\begin{multline*}
%\delta(\bcE_h(\bcX_h), \bcE(\bcX)) \leq \sup_{\bx_h\in \bcE_h(\bcX_h); \norm{\bx_h}=1}
%\norm{\bx_h - \Lambda_h^{-1} \bx_h}\\
% =  \sup_{\bx_h\in \bcE_h(\bcX); \norm{\bx_h}=1}
%\norm{\bx_h - \Lambda_h^{-1} \bx_h}_{DG} \leq C \delta^* (\bcE(\bcX) , \bcX_h).
%\end{multline*}
%\end{proof}

%\begin{theorem}\label{eigenEst}
%If $m$ is the multiplicity of an eigenvalue $\eta\neq 1$  of $\bT$ then there exists a constant $C>0$ independent
%of $h$ such that
%\[
%\disp\max_{1\leq i\leq m} |\eta - \eta_{i, h}|\leq  C \, 
%\delta^*(\bcE(\bcX), \bcX_h)^2.
%\]
%\end{theorem}

\begin{theorem}\label{errorE}
Assume that  $\bcE(\bcX) \subset \H^t(\bdiv,\O)\times \H^t(\O)^{\nxn}$, then there exists $C>0$ independent of $h$ and $\lambda$ such that 
\begin{equation}\label{eigenspace}
\displaystyle\gap(\bcE_h(\bcX_h), \bcE(\bcX))\leq \frac{C}{\dd_{\kappa}} h^{\min\{t, k\}}. 
\end{equation}
Moreover, there exists $C'>0$ independent of $h$ such that
\begin{equation}\label{eigenvalue}
\disp\max_{1\leq i\leq m} |\kappa - \kappa_{i, h}|\leq  \frac{C'}{{\dd}_{\kappa}} \, h^{2\min\{t, k\}}
\end{equation}
\end{theorem}
\begin{proof}
Using the estimate \eqref{aste} from the last theorem and proceeding as in the proof of \eqref{asymp} we immediately obtain  \eqref{eigenspace}.

Let $\kappa_{1, h}, \cdots, \kappa_{m, h}$ be the eigenvalues of 
$\bT_h:\, \bcX_h \to \bcX_h$  lying in  $D_\kappa$ and repeated according to their algebraic multiplicity.
 We denote by 
$\bx_{i,h}$ the eigenfunction  corresponding to $\kappa_{i,h}$ and satisfying 
$\norm{\bx_{i,h}}_{DG}=1$. We know from Theorem \ref{hatgap} that, if $h$ is sufficiently small,
\begin{equation*}
 \delta(\bx_{i,h},\bcE(\bcX))\leq \frac{C}{\dd_{\kappa}} \delta^*(\bcE(\bcX), \bcX_h).
\end{equation*}
Then, there exists an eigenfunction $\bx:=(\bsig, \br) \in \bcE(\bcX)$ satisfying
%Then, as $\bcE(\bcX)$ has finite dimension, there exists an eigenfunction $\bx:=(\bsig, \br) \in \bcE(\bcX)$ that satisfies (by virtue of \eqref{aste})
\begin{equation*}
\norm{\bx_{i,h}-\bx}_{DG} = \delta(\bx_{i,h},\bcE(\bcX)) \leq \gap(\bcE_h(\bcX_h), \bcE(\bcX)) \leq \frac{C}{\dd_\kappa} \delta^*(\bcE(\bcX), \bcX_h)\to 0 \quad \text{as $h\to 0$},
\end{equation*}
which proves that $\norm{\bx}_{DG}$ is bounded from below and above by constant independent of $h$.
Proceeding as in the proof of the consistency property in Theorem \ref{cea} we readily  obtain that
\begin{equation}\label{consist}
 A_h(\bx, \by_h) = \kappa B(\bx, \by_h) 
\end{equation}
for all $\by_h:=(\btau_h, \bs_h)\in \bcX_h$.
With the aid of \eqref{consist}, it is easy to show that the identity
\begin{equation*}
A_h(\bx-\bx_{i,h},\bx-\bx_{i,h})
- \kappa B(\bx-\bx_{i,h},\bx-\bx_{i,h})
=\left(\kappa_{i,h}-\kappa\right) B(\bx_{i,h},\bx_{i,h})
\end{equation*}
holds true. Now, according to Lemma~3.6 of \cite{MMR}, for any $\bx\in\bcE(\bcX)$, $\bx\neq 0$, it holds that $B(\bx,\bx)>0$.Thus, since  $\bcE(\bcX)$ is finite-dimensional,
there exists $c>0$, independent of $h$, such that $B(\bx,\bx)\geq c\norm{\bx}_{DG}$.
This proves that $B(\bx_{ih},\bx_{ih})\geq\frac{c}2$  for $h$ sufficiently
small. We obtain from \eqref{boundA1} that  
\begin{equation*}\label{eq11}
	\frac{c}2 \abs{\kappa_{i,h}-\kappa} \leq \abs{A_h(\bx-\bx_{i,h},\bx-\bx_{i,h})} + |\kappa| \abs{B(\bx-\bx_{i,h},\bx-\bx_{i,h})} \leq C (\norm{\bx-\bx_{i,h}}_{DG}^*)^2.
\end{equation*}
Since $\bx:=(\bsig,\br)$ and $\bx_{i,h}:=(\bsig_h,\br_h)$, and  by definition of  $\|\cdot\|_{DG}^{*}$ we have
\begin{equation*}
\norm{\bx-\bx_{i,h}}_{DG}^*:=\|(\bsig,\br)-(\bsig_h,\br_h)\|_{DG}^{*}=\|(\bsig,\br)-(\bsig_h,\br_h)\|_{DG}+\|h_{\cF}^{1/2}\mean{\bdiv(\bsig-\bsig_h)}\|_{\cF_h^{*}}.
\end{equation*}
It follows from Theorem \ref{conv}, Lemma \ref{TmenosTh} and the interpolation error estimates \eqref{asymp0}-\eqref{asymQ} that 
\begin{equation}\label{eq12}
	\|(\bsig,\br)-(\bsig_h,\br_h)\|_{DG}\leq C_0 \gap(\bcE(\bcX), \bcE_h(\bcX_h)) \leq C_1 h^{\min\{t, k\}} \Big( 1 + \norm{\bsig}_{\H^t(\bdiv, \O)} +  \norm{\br}_{\H^t(\O)^{\nxn}} \Big).
\end{equation}
On the other hand, 
\begin{equation}\label{fin1}
\|h_{\cF}^{1/2}\mean{\bdiv(\bsig-\bsig_h)}\|_{\cF_h^{*}}\leq\|h_{\cF}^{1/2}\mean{\bdiv(\bsig-\Pi_h\bsig)}\|_{\cF_h^{*}}+\|h_{\cF}^{1/2}\mean{\bdiv(\Pi_h\bsig-\bsig_h)}\|_{\cF_h^{*}}
\end{equation}
and it follows from \eqref{newE} that 
\begin{equation}\label{fin2}
	\|h_{\cF}^{1/2}\mean{\bdiv(\bsig-\Pi_h\bsig)}\|_{\cF_h^{*}} \leq C_2 h_K^{\min(t,k)} \norm{\bdiv \bsig}_{t,\O}.
\end{equation}
Finally,  using \eqref{discTrace}, \eqref{asympDiv} and \eqref{eq12} yield 
\begin{align}\label{fin3}
\begin{split}
\|h_{\cF}^{1/2}\mean{\bdiv(\Pi_h\bsig-\bsig_h)}\|_{\cF_h^{*}} &\leq C_3\|\bdiv(\Pi_h\bsig-\bsig_h)\|_{0,\O}
\\ &\leq C_3
\big(\|\bdiv(\Pi_h\bsig-\bsig)\|_{0,\O}+\|\bdiv_h(\bsig-\bsig_h)\|_{0,\O}\big)\\
& \leq C_3 \big(\|\bdiv(\Pi_h\bsig-\bsig)\|_{0,\O}+\|(\bsig,\br)-(\bsig_h,\br_h)\|_{DG}\big) \\&\leq C_4h^{\min\{t,k\}}\Big( 1 + \norm{\bsig}_{\H^t(\bdiv, \O)} +  \norm{\br}_{\H^t(\O)^{\nxn}} \Big).
\end{split}
\end{align}
Combining \eqref{eq11}, \eqref{fin1}-\eqref{fin3} and \eqref{eq12}, we obtain \eqref{eigenvalue}.
\end{proof}
 
\begin{remark}
In the proof provided above for the error estimate \eqref{eigenvalue} the constant $C'$ is not independent of $\lambda$. Indeed, according to the proof of Lemma~3.6 from \cite{MMR}, we have that 
\begin{equation*}
B((\bsig,\br),(\bsig,\br))=\int_{\O}\mathcal{C}^{-1}\bsig:\bsig\geq\min\left\{\frac{n}{n\lambda+2\mu},\frac{1}{2\mu}\right\}\norm{\bsig}^2_{0,\O}\geq 0.
\end{equation*}
Therefore,  the constant $c$ in the proof above tends to zero when $\lambda$ goes to infinity. However, the numerical experiments presented below suggest that \eqref{eigenvalue} holds true uniformly in $\lambda$. 
\end{remark}
\begin{remark}
We notice that there is in  \eqref{eigenspace} and \eqref{eigenvalue} a hidden reliance on $\lambda$ through the constant $\dd_{\kappa}:=\frac{1}{2}\dist\big(\kappa,\sp(\bT)\setminus\{\kappa\}\big)$ because $\sp(\bT)$ depends on $\lambda$. The constant $\dd_{\kappa}$ measures the deterioration of the error estimates given in Theorem \ref{errorE} when the eigenvalue $\kappa$ is too close to the accumulation point 0. 
\end{remark}

\begin{remark}\label{regEigenfun}
We point out that, thanks to Lemma \ref{reg}, we always have that $\bcE(\bcX) \subset \set{(\btau,\br) \in [\HsO^{\nxn}]^2:\, \bdiv \btau \in \H^1(\O)^n}$ for all $s\in (0,\ws)$. Consequently, the error estimates given in Theorem \ref{errorE} will always hold true for any $t\in(0,\ws)$ even if $\ws\leq 1/2$. However, it may happen that some eigenspaces satisfy the regularity assumption of the theorem with $t \geq \ws$. 
\end{remark}
\section{Numerical results}\label{NUMERICOS}

We present a series of numerical experiments to solve the elasticity eigenproblem in mixed form with the discontinuous Galerkin scheme \eqref{DGshort}.  All the numerical results have been obtained by using the FEniCS Problem Solving Environment \cite{fenics}. For simplicity we consider a two-dimensional model problem. We choose $\Omega = (0,1)\times(0,1)$,  $\rho = 1$, and a Young modulus $E=1$. We will let the Poisson ratio $\nu$ take different values in $(0,1/2]$. We recall that the Lam\'e coefficients are related to $E$ and  $\nu$ by
\begin{equation*}
\label{LAME}
\lambda:=\frac{E\nu}{(1+\nu)(1-2\nu)}
\qquad\text{and}\qquad 
\mu:=\frac{E}{2(1+\nu)}.
\end{equation*}
The limit problem corresponding to $\lambda=\infty$ is obtained by taking $\nu = 1/2$. In all our experiments we used uniform meshes with the symmetry pattern shown in Figure~\ref{FIG:MESH}. The refinement parameter $N$ represents the number of elements on each edge.

\begin{figure}[H]
\begin{center}
\begin{minipage}{6cm}
\centerline{\includegraphics[height=6cm, angle=-90]{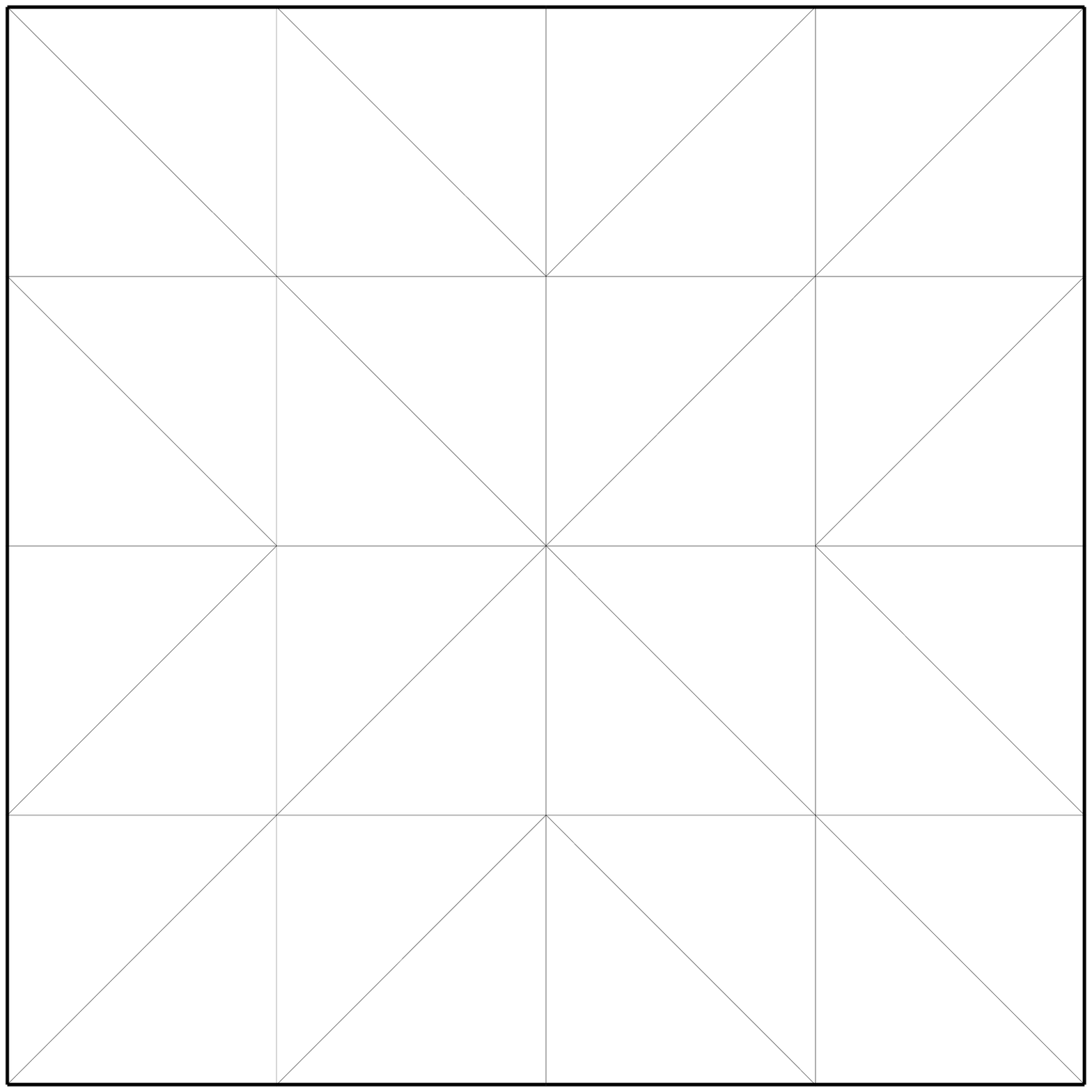}}
\centerline{$N=4$}
\end{minipage}
\begin{minipage}{6cm}
\centerline{\includegraphics[height=6cm, angle=-90]{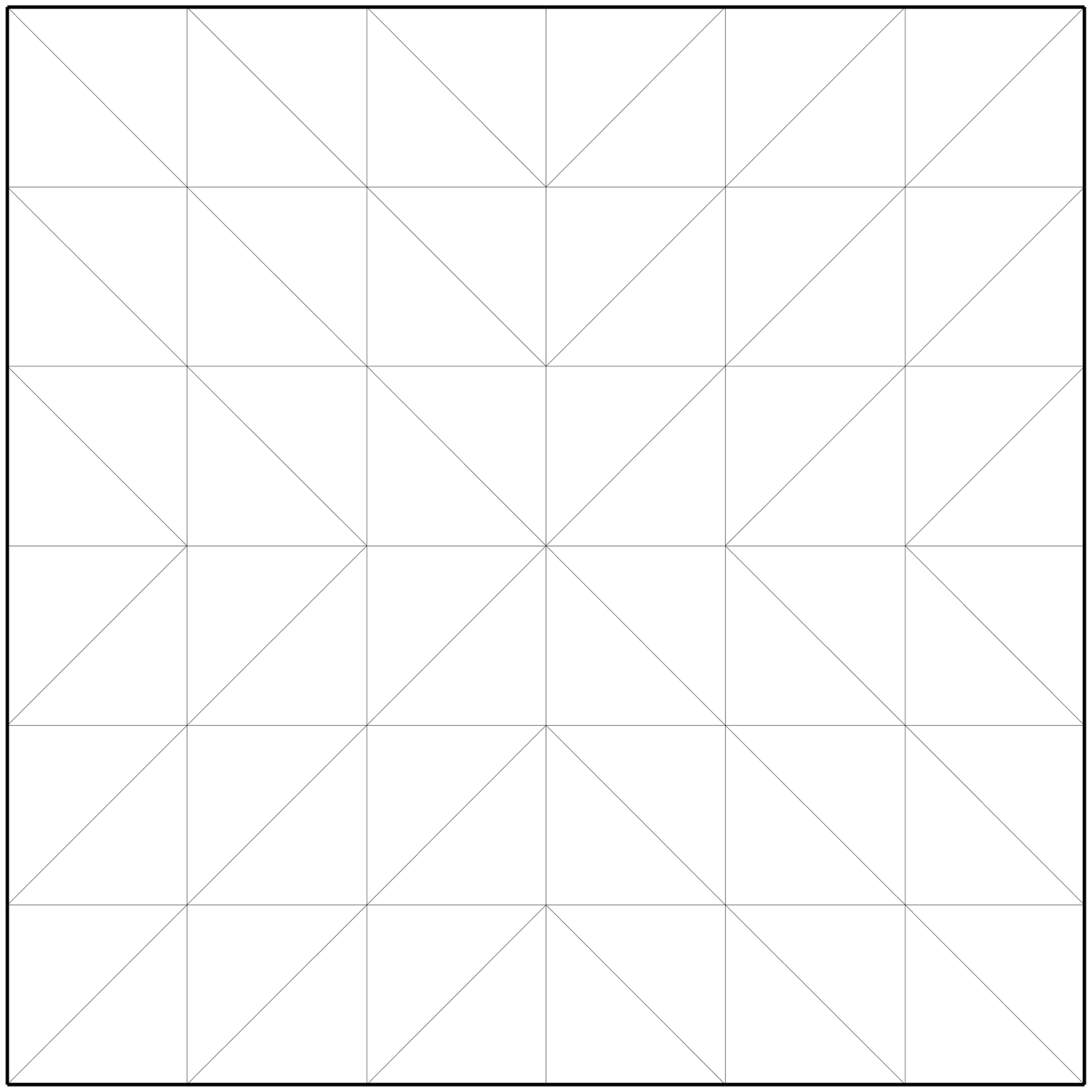}}
\centerline{$N=6$}
\end{minipage}
\caption{Uniform meshes}
\label{FIG:MESH}
\end{center}
\end{figure}

In the first tests we are concerned with the determination of a reliable  stabilization parameter $\texttt{a}_S$. We know that the spectral correctness of the method can only be guaranteed if $\texttt{a}_S$ is sufficiently large (Proposition \ref{infsupDh}) and if the meshsize $h$ is sufficiently small (cf. Remark \ref{rem2}). In a first stage, we fix the refinement level to $N=8$ and report in Tables \ref{TABLA1}, \ref{TABLA2} and \ref{TABLA3} the 10 smallest vibration frequencies computed for different values of $\texttt{a}_S$. The polynomial degrees are given by $k=3, 4, 5$, respectively. The boxed numbers are spurious eigenvalues. We observe that they emerge at random positions when we vary $\texttt{a}_S$ and $k$ and they disappear completely when $\texttt{a}_S$ is sufficiently large.

\begin{table}[H]
\footnotesize
\singlespacing
\begin{center}
\begin{tabular}{ c c c c c  }
\toprule
 $\texttt{a}_S=5$  &$\texttt{a}_S=10$ &$\texttt{a}_S=20$ & $\texttt{a}_S=40$ & $\texttt{a}_S=80$ \\ 
 \midrule
0.6804474&0.6804497&0.6804460&0.6804472& 0.6804472\\
1.6988814&1.6988904&1.6988615&1.6988797&1.6988800\\
1.8222056&1.8222073&1.8221859&1.8222050& 1.8222052  \\
2.9476938&2.9476927&\ffb{2.3856290}       &2.9476928& 2.9476933\\
3.0174161&3.0174530&\ffb{2.3862301}       &3.0174095& 3.0174114	\\
3.4432120&3.4432156&\ffb{2.5833172}       &3.4432158&3.4432168\\
4.1417685&4.1417626& \ffb{2.5839852}      &4.1417697&4.1417750 \\
4.6308354&4.6308072&2.9477062&4.6308465&4.6308549\\
4.7616007&4.7615186&3.0174627&4.7616237&4.7616317 \\
4.7879824&4.7879191&3.4432320&4.7880173&4.7880298\\
\bottomrule             
\end{tabular}
\end{center}
\caption{Vibration frequencies for $k=3$, $\nu=0.35$ and $N=8$}
\label{TABLA1}
\end{table}

\begin{table}[H]
\footnotesize
\singlespacing
\begin{center}
\begin{tabular}{ c c c c c c }
\toprule
 $\texttt{a}_S=5$  &$\texttt{a}_S=10$ &$\texttt{a}_S=20$ &  $\texttt{a}_S=40$ & $\texttt{a}_S=80$  \\ 
 \midrule
 
0.6805737&0.6805737&0.6805737&0.6805736&0.6805737\\
1.6990333&1.6990333&1.6990332&1.6990329&1.6990330\\
1.8222095&1.8222094&1.8222095&1.8222095&1.8222096\\
2.9476921&\ffb{2.2970057}       &2.9476922&2.9476922&2.9476922\\
3.0176437&\ffb{2.3909952}       &3.0176400&3.0176421&3.0176428\\
3.4432473&2.9476924&\ffb{3.1845593}     &3.4432470&3.4432472\\
4.1417687&3.0176452&\ffb{3.4392819}       &4.1417705&4.1417709\\
\ffb{4.5534365}       &3.4432480&3.4432839&4.6309421&4.6309433\\
4.6309432&4.1417718&4.1417737&4.7615808&4.7615812\\
\ffb{4.7195356}       &4.6309455&4.6309470&4.7882380&4.7882400\\

\bottomrule             
\end{tabular}
\end{center}
\caption{Vibration frequencies for $k=4$, $\nu=0.35$ and $N=8$}
\label{TABLA2}
\end{table}

\begin{table}[H]
\footnotesize
\singlespacing
\begin{center}
\begin{tabular}{  c c c c c c }
\toprule
 $\texttt{a}_S=5$  &$\texttt{a}_S=10$ &$\texttt{a}_S=20$ & $\texttt{a}_S=40$& $\texttt{a}_S=80$   \\ 
 \midrule

0.6806522&0.6806522&0.6806522&0.6806522&0.6806522\\
1.6991254&1.6991254&1.6991255&1.6991250&1.6991253\\
1.8222137&1.8222137&1.8222138&1.8222137&1.8222137\\
2.9476935&2.9476935&\ffb{2.4714299}   &2.9476935&2.9476935\\
3.0177848&3.0177848&\ffb{2.4822317}   &3.0177827&3.0177844\\
3.4432656&3.4432656&2.9476935&3.4432652&3.4432656\\
4.1417853&4.1417852&3.0177862&4.1417845&4.1417852\\
4.6310201&4.6310201&3.4432657&4.6310172&4.6310196\\
4.7615803&4.7615803&4.1417853&4.7615800&4.7615802\\
4.7883889&4.7883889&4.6310208&4.7883835&4.7883878\\

\bottomrule             
\end{tabular}
\end{center}
\caption{Vibration frequencies for $k=5$, $\nu=0.35$ and $N=8$}
\label{TABLA3}
\end{table}
Next, we present in Table \ref{TABLA4} different approximations of the first 10 vibration frequencies corresponding to $N = 8, 16, 32, 64$, obtained with $\texttt{a}_S =20$ and a polynomial degree $k = 3$.   We notice that as the level of refinement increases the lower frequencies are progressively cleaned from spurious modes. We conclude that our method provides a correct approximation of the spectrum as long as $N$ and $\texttt{a}_S$ are large enough. In the forthcoming tests we will take $\texttt{a}_S=1000$. We point out that the previous tests have been carried out with a Poisson ratio $\nu=0.35$, but similar results were obtained for values ranging from 0.35 to 0.5. 
\begin{table}[H]
\footnotesize
\singlespacing
\begin{center}
\begin{tabular}{c c c c  }
\toprule
$N=8$             &  $N=16$         &   $N=32$         & $N=64$   \\ 
 \midrule

0.6804460& 0.6806838 & 0.6807775& 0.6808142 \\
1.6988615& 1.6991595 & 1.6992689& 1.6993109  \\
1.8221859& 1.8222154 & 1.8222207& 1.8222228  \\
\ffb{2.3856290 }      & 2.9476935 & 2.9476956& 2.9476963 \\
\ffb{2.3862301}       & 3.0178279 & 3.0180082& 3.0180748\\
\ffb{2.5833172}       &\ffb{3.2760743} & 3.4432923& 3.4433002 \\
\ffb{2.5839852 }      &\ffb{3.2777582} & 4.1418082& 4.1418158\\
2.9477062&3.4432656 & \ffb{4.4519274}       & 4.6311877 \\
3.0174627&\ffb{3.5133204} &\ffb{4.4548953}       & 4.7615817\\
3.4432320&\ffb{3.5153213} & 4.6311437& 4.7886836\\
\bottomrule             
\end{tabular}
\end{center}
\caption{Vibration frequencies for $k=3$, $\texttt{a}_S=20$ , $\nu=0.35$ and different refinement levels}
\label{TABLA4}
\end{table}

The subsequent numerical tests are aimed to determine the convergence rate of the scheme. With the boundary conditions considered in our model problem, it turns out that (cf. \cite{MMR} and the references therein) the regularity exponents $\ws$ defined in Lemma \ref{reg} are given by Table \ref{TABLA5} for different values of the Poisson ratio $\nu$. 

\begin{table}[H]
\footnotesize
\singlespacing
\begin{center}
\begin{tabular}{c c}
\toprule
 $\nu$ & $\ws$  \\ 
 \midrule
0.35 & 0.6797\\
0.49&0.5999\\
0.5&0.5946\\
\bottomrule             
\end{tabular}
\end{center}
\caption{Sobolev regularity exponents}
\label{TABLA5}
\end{table}

We present in Tables \ref{TABLA22.0}, \ref{TABLA22.1} and \ref{TABLA22.2} (corresponding to the polynomial degrees $k=2,3,4$, respectively) the first two vibration frequencies computed on a series of nested  meshes for a range of Poisson ratios given by $\nu=0.35, 0.49, 0.5$. We also report in these tables  an estimate  of the order of convergence $\alpha$, as well as more accurate approximations of the vibration frequencies obtained by means of the least-squares fitting technique explained in \cite[Section 6]{MMR}.  Comparing with the exponents given in Table \ref{TABLA5}, we observe that our method provides a double order of convergence for the vibration frequencies. Namely, in all cases  we have $\alpha\simeq 2\ws$, which corresponds to the the worst possible order of convergence. The eigenfunctions corresponding to higher natural frequencies are oscillating  but they can be more regular (see Remark \ref{regEigenfun}), which justifies the use of high polynomial orders of approximation.  Finally, we point out that the method is clearly locking-free.

\begin{table}[H]
\footnotesize
\singlespacing
\begin{center}
\begin{tabular}{c |c c c c |c| c }
\toprule
 $\nu $        & $N=16$             &  $N=32$         &   $N=48$         & $N=64$ & $\alpha$ &$\lambda_{ex}$ \\ 
 \midrule
 \multirow{2}{1cm}{0.35}   &0.6806068 &   0.6807467&    0.6807850 &   0.6808020 &  1.34  &   0.6808381 \\
                                         &1.6990672 &  1.6992327 &     1.6992773&   1.6992969 &  1.37  & 1.6993373  \\
          \hline

\multirow{2}{1cm}{0.49}    &0.6987402&    0.6991833&  0.6993160&    0.6993779&    1.19&     0.6995295      \\
                                         &1.8359946&   1.8366760 &  1.8368781&   1.8369722 &   1.20 &  1.8372009      \\
          \hline

\multirow{2}{1cm}{0.5}     &0.7007298 &   0.7012091&  0.7013534&    0.7014210  &  1.18 &    0.7015881   \\
                                        &1.8472390 &  1.8479824 &  1.8482043&   1.8483081   & 1.19  & 1.8485623   \\
       
\bottomrule             
\end{tabular}
\end{center}
\caption{Lowest vibration frequencies for $k=2$, $\texttt{a}_S=1000$ and convergence order}
\label{TABLA22.0}
\end{table}

\begin{table}[H]
\footnotesize
\singlespacing
\begin{center}
\begin{tabular}{c |c c c c |c| c }
\toprule
 $\nu $        & $N=16$             &  $N=32$         &   $N=48$         & $N=64$ & $\alpha$& $\lambda_{ex}$  \\ 
 \midrule
 \multirow{2}{1cm}{0.35}           &0.6806839&   0.6807775 &  0.6808029&   0.6808142&   1.35 &   0.6808379\\
                                                 &1.6991607&   1.6992690 &  1.6992981&   1.6993109&    1.37&   1.6993373 \\
          
          \hline

\multirow{2}{1cm}{0.49}       &0.6989872 &  0.6992929 &  0.6993836 &  0.6994258 &  1.20&   0.6995284 \\
                                            &1.8363810 &  1.8368436 &  1.8369810 &  1.8370450 &  1.20&   1.8372002 \\
          
\hline

\multirow{2}{1cm}{0.5}         &0.7009977 &  0.7013286 &  0.7014275  & 0.7014736 &  1.19 &  0.7015868 \\
                                            &1.8476611  &  1.8481669 &  1.8483181  & 1.8483888 &  1.19 &  1.8485618 \\
       
\bottomrule             
\end{tabular}
\end{center}
\caption{Lowest vibration frequencies for $k=3$, $\texttt{a}_S=1000$ and convergence order}
\label{TABLA22.1}
\end{table}

\begin{table}[H]
\footnotesize
\singlespacing
\begin{center}
\begin{tabular}{c |c c c c |c| c }
\toprule
 $\nu $        & $N=16$             &  $N=32$         &   $N=48$         & $N=64$ & $\alpha$& $\lambda_{ex}$  \\ 
 \midrule
 \multirow{2}{1cm}{0.35} &0.6807342&   0.6807973&   0.6808144&   0.6808219&   1.36  & 0.6808376 \\
                                       &1.6992195&   1.6992917&   1.6993112 &  1.6993198 &   1.36  & 1.6993377 \\
          
          \hline

\multirow{2}{1cm}{0.49} & 0.6991499&   0.6993638 &  0.6994272 &  0.6994567 &  1.20&   0.6995284  \\
                                      & 1.8366280&   1.8369510 &  1.8370470 &  1.8370917 &  1.20&   1.8372000  \\
          
\hline

 \multirow{2}{1cm}{0.5} &0.7011738&  0.7014060 &  0.7014751  & 0.7015075 &  1.19&   0.7015869   \\
                                     &1.8479310&  1.8482851 &  1.8483911  & 1.8484407 &  1.19&   1.8485618   \\
       \bottomrule             
\end{tabular}
\end{center}
\caption{Computed lowest vibration frequencies for $k=4$, $\texttt{a}_S=1000$ and convergence order}
\label{TABLA22.2}
\end{table}
\section{Appendix. The limit problem}
\label{APPendix}
As was shown in the previous section, the proposed method works fine also for the limit problem ($\lambda=+\infty$), namely, for perfectly incompressible elasticity. In this appendix, we will establish a spectral characterization in this case. Also, we will prove that the eigenvalues of the nearly incompressible elasticity problem converge to those of the incompressible elasticity problem as $\lambda\rightarrow \infty$.

In the limit case $\lambda=+\infty$, the bilinear forms $A$ and $B$ change in their definitions, since the term where $\lambda$ appears in \eqref{invcCop} vanishes. Therefore,  the limit eigenvalue problem reads as follows: Find $\kappa\in\mathbb{R}$ and $(\bsig,\br)\in\bcW\times\bcQ$ such that
\begin{equation}\label{limit0}
A_{\infty}((\bsig,\br),(\btau,\bs))=\kappa B_{\infty}((\bsig,\br),(\btau,\bs))\qquad\forall(\btau,\bs)\in\bcW\times\bcQ
\end{equation}
with
\[
B_{\infty}((\bsig,\br),(\btau,\bs)):=\frac{1}{2\mu}\int_{\O}\bsig^{\tD}:\btau^{\tD}+\int_{\O}\br:\btau+\int_{\O}\bs:\bsig
\]
and
\[
A_{\infty}((\bsig,\br),(\btau,\bs)):=\int_{\O}\rho^{-1} \bdiv \bsig \cdot \bdiv \btau +B_{\infty}((\bsig,\br),(\btau,\bs))
\]
for all $(\bsig,\br),(\btau,\bs)\in\bcW\times\bcQ.$

It is easy to check that $A_{\infty}$ is a bounded bilinear form. Moreover, the arguments used in the proofs of Propositions~\ref{normEquiv} and \ref{infsupA-cont} hold true for $\lambda=+\infty$, so that  $A_{\infty}$ satisfies the following inf-sup condition:
\begin{equation*}
  \sup_{\taus\in \bcW\times \bcQ} \frac{A_{\infty}((\bsig,\br),(\btau,\bs))}{\norm{\taus}} \geq \alpha \norm{(\bsig,\br)}\qquad \forall(\bsig_,\br) \in 
  \bcW\times \bcQ.
 \end{equation*}
In consequence, we are in a position to introduce a solution operator for the limit eigenvalue problem. Let $\bT_{\infty}: [\L^2(\O)^{\nxn}]^2 \to \bcW\times \bcQ$ be
defined for any $(\bF, \bg) \in [\L^2(\O)^{\nxn}]^2$  by
\begin{equation*}\label{PLIM}
A_\infty(\bT_{\infty}(\bF,\bg),(\btau,\bs))= B_\infty((\bF,\bg),(\btau,\bs))\qquad\forall(\btau,\bs)\in\bcW\times\bcQ.
\end{equation*}

It is easy to check that $\mu$ is a non-zero eigenvalue of $\bT_{\infty}$ with eigenfunction $(\bsig_{\infty},\br_{\infty})$ if and only of $\kappa = 1/\mu$ is a non-vanishing eigenvalue of problem \eqref{limit0} with the same eigenfunction.

Our first goal is to prove that the operators $\bT$ defined by \eqref{charcT} converges to $\bT_\infty$ as $\lambda$ goes to infinity. To recall that $\bT$ actually depends on $\lambda$,   in what follows we will denote it by $\bT_\lambda$.

Before proving the convergence of $\bT_{\lambda}$ to $\bT_{\infty}$, we will characterize the spectrum of $\bT_{\infty}$. 
Let $\bcK$ be defined as in \eqref{K} and 
\begin{equation*}
[\bcK\times\bcQ]^{\bot_{B_{\infty}}}
:=\left\{ (\bsig,\br)\in\bcW\times\bcQ:
\ B_{\infty}( (\bsig,\br), (\btau,\bs))=0 \quad \forall (\btau,\bs)\in\bcK\times\bcQ\right\}.
\end{equation*}
We observe that $\bT_{\infty}|_{\bcK\times\bcQ}:\bcK\times\bcQ\rightarrow\bcK\times\bcQ$ reduces to the identity, so that $\mu=1$ is  an eigenvalue of $\bT_{\infty}$. Moreover, its associated eigenspace is precisely $\bcK\times\bcQ$.

Let us introduce the following operator which will play a role similar to that of $\bP$ in the limit problem:
\begin{align*}
\bP_{\infty}:\bcW\times\bcQ&\rightarrow\bcW\times\bcQ,\\
      (\bsig,\br)&\mapsto \bP_{\infty}\bsig:=(\widetilde{\bsig}, \widetilde{\br}).
\end{align*}
where $(\wbsig,(\wbu,\wbr))\in\bcW\times[\L^2(\O)^n\times\bcQ]$ is the solution of the following problem:
\begin{align}
\frac{1}{2\mu}&\int_{\O}\wbsig^{\tD}:\btau^{\tD}+\int_{\O}\wbu\cdot\bdiv\btau+\int_{\O}\btau:\wbr=0\qquad\forall\btau\in\bcW,\label{eq1}\\
&\int_{\O}\bv\cdot\bdiv\wbsig+\int_{\O}\wbsig:\bs=\int_{\O}\bv\cdot\bdiv\bsig\qquad\forall(\bv,\bs)\in\L^2(\O)^n\times\bcQ.\label{eq2}
\end{align}

The previous problem is well posed, since the ellipticity of $\int_{\O}\bsig^{\tD}:\btau^{\tD}$ in the corresponding kernel is established in Lemma 2.3 of \cite{MMR-stokes} and the following inf-sup condition holds true (see \cite{BoffiBrezziFortin}):

\[
\displaystyle\sup_{\btau\in\bcW}\frac{\int_{\O}\bv\cdot\bdiv\btau+\int_{\O}\bs:\btau}{\norm{\btau}_{\H(\bdiv,\O)}}\geq\beta(\norm{\bv}_{0,\O}+\norm{\bs}_{0,\O})\qquad\forall(\bv,\bs)\in\L^2(\O)^n\times\bcQ.
\]

We observe that problem \eqref{eq1}--\eqref{eq2} is a dual mixed formulation with weakly imposed symmetry
of the following incompressible elasticity problem with volumetric force density $-\bdiv\bsig$
\begin{align}
-\bdiv\wbsig&=-\bdiv\bsig\hspace{0.85cm}\text{in}\, \O,\label{strong1}\\
      \frac{1}{2\mu}\wbsig^{\tD}&=\boldsymbol{\varepsilon}(\wbu)\hspace{1.49cm}\text{in}\, \O,\label{stron2}\\
        \wbsig\bn&=\0\hspace{2.10cm}\text{on}\, \Gamma_N,\label{strong3}\\
        \wbu&=\0\hspace{2.10cm}\text{on}\, \Gamma_D.\label{strong4}
\end{align}

It is easy to check that $(\wbsig,\wbu)\in\H(\bdiv,\O)\times \H^1(\O)^n$ satisfies \eqref{strong1}--\eqref{strong4} if and only if $(\wbsig,(\wbu,\wbr))\in\bcW\times[\L^2(\O)^n\times\bcQ]$ is the solution of \eqref{eq1}--\eqref{eq2} with $\wbr=\frac{1}{2}[\nabla\wbu-(\nabla\wbu)^{\t}].$

Now, by resorting to the relation between the incompressible elasticity and the Stokes problems, we conclude that there exists $\ws_{\infty}\in (0,1)$ depending only on $\O$ and $\mu$ (see for instance \cite{girault-raviart}) such that, for all $s\in (0,\ws_{\infty})$  the solution $\wbu$ of \eqref{strong1}--\eqref{strong4} belongs to $\H^{1+s}(\O)^n$ and the following estimate hold true 
\begin{equation*}\label{regutilde}
\norm{\wbu}_{1+s,\O}\leq C\norm{\bdiv\bsig}_{0,\O},
\end{equation*}
with a constant $C$ independent of $\bsig$.

The following lemma is a consequence of this regularity result.
\begin{lemma}\label{reginfty}
For all $s\in (0,\ws)$ and $(\bsig,\br)\in\bcW\times\bcQ$, if $(\wbsig,(\wbu,\wbr))$ is the solution of \eqref{eq1}--\eqref{eq2}, then  $\widetilde{\bsig}\in \H^s(\O)^{n\times n}$, $\widetilde{\bu}\in\H^{1+s}(\O)^{n\times n}$, $\wbr\in\H^s(\O)^{n\times n}$ and 
\begin{equation*}
\norm{\wbsig}_{s,\O}+\norm{\wbu}_{1+s,\O}+\norm{\wbr}_{s,\O}\leq C\norm{\bdiv\bsig}_{0,\O},
\end{equation*}
with a constant $C$ independent of $\bsig$. Consequently, $\bP_{\infty}(\bcW\times\bcQ)\subset\H^{s}(\O)^{n\times n}\times\H^{s}(\O)^{\nxn}.$
\end{lemma}

We observe that $\bP_{\infty}$ is idempotent  and that $\ker(\bP_{\infty})=\bcK\times\bcQ$. Moreover, being $\bP_{\infty}$  a projector, the orthogonal decomposition $\bcW\times\bcQ=(\bcK\times\bcQ)\oplus\bP_{\infty}(\bcW\times\bcQ)$ holds true. On the other hand, $\bP_{\infty}(\bcW\times\bcQ)$ is an invariant space of $\bT_{\infty}$ (see Proposition A.1 in \cite{MMR}).  

\begin{prop}\label{regP}
For all $s\in (0,\ws)$ 
\begin{equation}\label{invsubset}
\bT_{\infty}(\bP_{\infty}(\bcW\times\bcQ))\subset\set{(\bsig^{*},\br^{*})\in\H^{s}(\O)^{\nxn}\times\H^{s}(\O)^{\nxn}:\, \bdiv\bsig^{*}\in\H^1(\O)^n},
\end{equation}
and there exists $C>0$ such that for all $(\bF,\bg)\in\bP_{\infty}(\bcW\times\bcQ)$, if $(\bsig^{*},\br^{*})=\bT_{\infty}(\bF,\bg)$, then
\begin{equation}\label{regstar}
\norm{\bsig^{*}}_{s,\O}+\norm{\bdiv\bsig^{*}}_{1,\O}+\norm{\br^{*}}_{s,\O}\leq C\norm{(\bF,\bg)}.
\end{equation}
Moreover, $\bT_{\infty}|_{\bP_{\infty}(\bcW\times\bcQ)}:\bP_{\infty}(\bcW\times\bcQ)\rightarrow \bP_{\infty}(\bcW\times\bcQ)$ is a compact operator.
\end{prop}
\begin{proof}
Let $(\bF,\bg)\in\bP_{\infty}(\bcW\times\bcQ)$ and $(\bsig^{*},\br^{*})=\bT_{\infty}(\bF,\bg).$ Hence, we have 
\begin{align*}
&\int_{\O}\rho^{-1}\bdiv\bsig^*\cdot\bdiv\btau+\frac{1}{2\mu}\int_{\O}\bsig^{*\tD}:\btau^{\tD}+\int_{\O}\br^*:\btau=\frac{1}{2\mu}\nonumber\int_{\O}\bF^{\tD}:\btau^{\tD}+\int_{\O}\bg:\btau\quad\forall\btau\in\bcW,\\
&\int_{\O}\bsig^*:\bs=\int_{\O}\bF:\bs\qquad\forall\bs\in\bcQ.
\end{align*}
Then, testing the first equation of the system above with $\btau\in\mathcal{D}(\O)^{\nxn}\subset\bcW$, we have that
\begin{equation*}
-\rho^{-1}\nabla(\bdiv\bsig^{*})+\frac{1}{2\mu}\bsig^{*\tD}+\br^{*}=\frac{1}{2\mu}\bF^{\tD}+\bg.
\end{equation*}
Hence, since $\rho$ and $\mu$ are constants, we conclude that $\bdiv\bsig^{*}\in\H^1(\O)^n.$

Since $\bP_{\infty}(\bcW\times\bcQ)$ is invariant with  respect to $\bT_{\infty}$, applying  Lemma \ref{reginfty} we obtain directly \eqref{invsubset}. On the other hand, \eqref{regstar} is a consequence of Lemma \ref{reginfty}. Finally,  the compactness of $\bT_{\infty}|_{\bP_{\infty}(\bcW\times\bcQ)}$ is a consequence of the following compact embedding
$$\set{(\bsig^{*},\br^{*})\in\H^{s}(\O)^{\nxn}\times\H^{s}(\O)^{\nxn}:\bdiv\bsig^{*}\in\H^1(\O)^n}\hookrightarrow\bcW\times\bcQ,$$
which allow us to conclude the proof.
\end{proof} 
Now we are in position to establish a spectral characterization for             $\bT_{\infty}.$
\begin{theorem}
The spectrum of $\bT_{\infty}$ decomposes as follows: $\sp(\bT_{\infty})=\set{0,1}\cup\set{\mu_k}_{k\in\mathbb{N}}$, where:
\begin{itemize}
\item[(i)] $\mu=1$ is an infinite-multiplicity eigenvalue of $\bT_{\infty}$ and its associated eigenspace is $\bcK\times \bcQ.$
\item[(ii)] $\mu=0$ is an eigenvalue of $\bT_{\infty}$ and its associated eigenspace is $\mathcal{Z}\times\bcQ$, where
$$\mathcal{Z}:=\set{\btau\in\bcW:\hspace{0.2cm}\btau^{\tD}=0}=\set{q\bI:\hspace{0.2cm}q\in\H^1(\O)\hspace{0.2cm}\text{and}\hspace{0.2cm}q=0\quad\text{on}\hspace{0.2cm}\Gamma_N}.$$
\item[(iii)] $\set{\mu_k}_{k\in\mathbb{N}}\subset(0,1)$ is a sequence of nondefective finite-multiplicity eigenvalues of $\bT_{\infty}$ which converge to zero and the corresponding eigenspaces lie in $\bP_\infty(\bcW\times\bcQ)$.
\end{itemize}
\end{theorem}
\begin{proof}
It is enough to follow the steps of Theorem 3.5 from \cite{MMR-stokes}.
\end{proof}

Now we are in position to establish the following convergence result.
\begin{lemma}\label{LIMITE}
There exists a constant $C>0$ such that 
$$\|(\bT_{\lambda}-\bT_{\infty})((\bF,\bg))\|\leq \frac{C}{\lambda}\|(\bF,\bg)\|_{0,\O}\qquad\forall(\bF,\bg)\in[\L^2(\O)^{\nxn}]^2.$$
\end{lemma}
\begin{proof}
Let $(\bF,\bg)\in[\L^2(\O)^{\nxn}]^2$ and let $(\bsig_{\lambda},\br_{\lambda}):=\bT_{\lambda}(\bF,\bg)$ and $(\bsig_{\infty},\br_{\infty}):=\bT_{\infty}(\bF,\bg)$. Then, from \eqref{charcT} and the definition of $\cC$ we have
\begin{align*}
\int_{\O}\rho^{-1}\bdiv\bsig_{\lambda}\cdot\bdiv\btau+\frac{1}{2\mu}\int_{\O}\bsig_{\lambda}^{\tD}:\btau^{\tD}+&\frac{1}{n(n\lambda+2\mu)}\int_{\O}\tr(\bsig_{\lambda})\tr(\btau)+\int_{\O}\br_{\lambda}:\btau\\
&=\frac{1}{2\mu}\int_{\O}\bF^{\tD}:\btau^{\tD}+\frac{1}{n(n\lambda+2\mu)}\int_{\O}\tr(\bF)\tr(\btau)+\int_{\O}\bg:\btau,\\
\int_{\O}\bsig_{\lambda}:\bs=\int_{\O}\bF:\bs.&
\end{align*}
Whereas
\begin{align*}
&\int_{\O}\rho^{-1}\bdiv\bsig_{\infty}\cdot\bdiv\btau+\frac{1}{2\mu}\int_{\O}\bsig_{\infty}^{\tD}:\btau^{\tD}+\int_{\O}\br_{\infty}:\btau=\frac{1}{2\mu}\nonumber\int_{\O}\bF^{\tD}:\btau^{\tD}+\int_{\O}\bg:\btau\quad\forall\btau\in\bcW,\\
\nonumber&\int_{\O}\bsig_{\infty}:\bs=\int_{\O}\bF:\bs\qquad\forall\bs\in\bcQ.
\end{align*}
Subtracting the  above equations we have 
\begin{align}
\nonumber\int_{\O}\rho^{-1}&\bdiv(\bsig_{\lambda}-\bsig_{\infty})\cdot\bdiv\btau+\frac{1}{2\mu}\int_{\O}(\bsig_{\lambda}^{\tD}-\bsig_{\infty}^{\tD}):\btau^{\tD}\\
&+\int_{\O}(\br_{\lambda}-\br_{\infty}):\btau
=\frac{1}{n(n\lambda+2\mu)}\int_{\O}\tr(\bF-\bsig_{\lambda})\tr(\btau)\qquad\forall\btau\in\bcW,\label{mix1}\\
\int_{\O}(\bsig_{\lambda}-&\bsig_{\infty}):\bs=0\qquad\forall\bs\in\bcQ.\label{simetrico}
\end{align}
Testing this equation with $\btau:=\bsig_{\lambda}-\bsig_{\infty}$ and $\bs:=\br_{\lambda}-\br_{\infty}$ we have 
%\begin{align*}
%\int_{\O}\rho^{-1}|\bdiv(\bsig_{\lambda}-\bsig_{\infty})|^2+\frac{1}{2\mu}\int_{\O}|\bsig_{\lambda}^{\tD}-\bsig_{\infty}^{\tD}|^2+\frac{1}{n(n\lambda+2\mu)}&\int_{\O}\tr(\bsig_{\lambda})\tr(\bsig_{\lambda}-\bsig_{\infty})\\
%=&\frac{1}{n(n\lambda+2\mu)}\int_{\O}\tr\bF\tr(\bsig_{\lambda}-\bsig_{\infty}).
%\end{align*}
%Then,
\begin{align*}
\rho^{-1}\|\bdiv(\bsig_{\lambda}-\bsig_{\infty})\|_{0,\O}^2+\frac{1}{2\mu}\|\bsig_{\lambda}^{\tD}-\bsig_{\infty}^{\tD}\|_{0,\O}^2 =&\frac{1}{n(n\lambda+2\mu)}\int_{\O}(\tr(\bF)-\tr(\bsig_{\lambda}))\tr(\bsig_{\lambda}-\bsig_{\infty})\\
    \leq&\frac{1}{n(n\lambda+2\mu)}\int_{\O}\|\tr(\bF)-\tr(\bsig_{\lambda})\|_{0,\O}\,\|\tr(\bsig_{\lambda}-\bsig_{\infty})\|_{0,\O}\\
     \leq&\frac{1}{n\lambda+2\mu}(\|\bF\|_{0,\O}+\|\bsig_{\lambda}\|_{0,\O})\|\bsig_{\lambda}-\bsig_{\infty}\|_{0,\O}\\
     \leq&\frac{C}{n\lambda}\|(\bF,\bg)\|_{0,\O}\|\bsig_{\lambda}-\bsig_{\infty}\|_{0,\O},
\end{align*}
where we have used \eqref{bT} to bound $\norm{\bsig_{\lambda}}_{0,\O}$. Moreover
\begin{equation*}
\underbrace{\min\left\{\rho^{-1},\frac{1}{2\mu}\right\}}_{C_{\rho,\mu}}\left(\|\bdiv(\bsig_{\lambda}-\bsig_{\infty})\|_{0,\O}^2+\|\bsig_{\lambda}^{\tD}-\bsig_{\infty}^{\tD}\|_{0,\O}^2\right) \leq\frac{C}{n\lambda}\|(\bF,\bg)\|_{0,\O}\|\bsig_{\lambda}-\bsig_{\infty}\|_{0,\O}.
\end{equation*}
 We observe that  $(\bsig_{\lambda}-\bsig_{\infty})\in\bcW$ is symmetric due to equation \eqref{simetrico}. Then, we resort to the following estimate (see \cite{BoffiBrezziFortinBook} for instance)
\begin{equation*}
C\|\bsig_{\lambda}-\bsig_{\infty}\|_{0,\O}^2\leq \|\bsig_{\lambda}^{\tD}-\bsig_{\infty}^{\tD}\|_{0,\O}^2+\|\bdiv(\bsig_{\lambda}-\bsig_{\infty})\|_{0,\O}^2 
\end{equation*}
with $C>0$ to deduce that
\begin{equation*}
C\|\bsig_{\lambda}-\bsig_{\infty}\|_{\H(\bdiv,\O)}\leq( \|\bsig_{\lambda}^{\tD}-\bsig_{\infty}^{\tD}\|_{0,\O}^2+\|\bdiv(\bsig_{\lambda}-\bsig_{\infty})\|_{0,\O}^2)^{1/2}.
\end{equation*}
Hence 
\begin{align}
\nonumber\|\bdiv(\bsig_{\lambda}-\bsig_{\infty})\|_{0,\O}^2 +\|&\bsig_{\lambda}^{\tD}-\bsig_{\infty}^{\tD}\|_{0,\O}^2 \\
&\leq\frac{C_{\rho,\mu}}{n\lambda}\|(\bF,\bg)\|_{0,\O}( \|\bsig^{\tD}-\bsig_{\infty}^{\tD}\|_{0,\O}^2+\|\bdiv(\bsig_{\lambda}-\bsig_{\infty})\|_{0,\O}^2)^{1/2}\label{root}
\end{align}
and, finally,
\begin{equation}\label{cotaf}
\|\bsig_{\lambda}-\bsig_{\infty}\|_{\H(\bdiv,\O)}\leq\frac{C}{\lambda}\|(\bF,\bg)\|_{0,\O},
\end{equation}
with $C$  a positive constant depending on $\rho$, $\mu$ and $n$.

On the other hand, taking into account the inf-sup condition \eqref{inSupbeta},  \eqref{mix1},  Cauchy-Schwarz inequality, \eqref{root} and \eqref{cotaf}, we have 
 \begin{align}
\nonumber\displaystyle\beta &\norm{\br_{\lambda}-\br_{\infty}}_{0,\O}\\&\leq\sup_{\btau\in \bcW} \frac{\frac{1}{n(n\lambda+2\mu)}\int_{\O}\tr(\bsig_{\lambda}-\bsig_{\infty})\tr(\btau)-\int_{\O}\rho^{-1}\bdiv(\bsig_{\lambda}-\bsig_{\infty})\cdot\bdiv\btau-\frac{1}{2\mu}\int_{\O}(\bsig_{\lambda}^{\tD}-\bsig^{\tD}_{\infty}):\nonumber\btau^{\tD}}{\norm{\btau}_{\HdivO}}\\
\nonumber&\leq\sup_{\btau\in \bcW}\frac{\frac{C}{n\lambda+2\mu}\|(\bF,\bg)\|_{0,\O}\|\btau\|_{0,\O}+\rho^{-1}\|\bdiv(\bsig_{\lambda}-\bsig_{\infty})\|_{0,\O}\|\bdiv\btau\|_{0,\O}+\frac{1}{2\mu}\|\bsig_{\lambda}^{\tD}-\bsig^{\tD}_{\infty}\|_{0,\O}\|\btau^{\tD}\|_{0,\O}}{\norm{\btau}_{\HdivO}}\\
&\leq\frac{C}{\lambda}\|(\bF,\bg)\|_{0,\O}\label{cotafg}.
\end{align}

Hence, the proof follows by combining \eqref{cotaf} and \eqref{cotafg}.

\end{proof}
Now we are in a position to establish the following result.
\begin{theorem}
Let $\mu_{\infty}>0$ be an eigenvalue of $\bT_{\infty}$ of multiplicity $m$. Let $D$ be any disc of the complex plane centered at  $\mu_{\infty}$ and containing no other element of the spectrum of $\bT_{\infty}.$ Then, for $\lambda$ large enough, $D$ contains exactly $m$ eigenvalues of $\bT_{\lambda}$ (repeated according to their respective multiplicities). Consequently, each eigenvalue $\mu_{\infty}>0$ of $\bT_{\infty}$ is a limit of eigenvalues $\mu$ of $\bT_{\lambda}$, as $\lambda$ goes to infinity.
\end{theorem}

%\renewcommand{\abstractname}{Acknowledgements}
%\begin{abstract}
%\section*{Acknowledgements}
%F. Lepe was supported by  Proyecto Plan Plurianual 2016-2020, Universidad del B\'io-B\'io, Chile.
%S. Meddahi was supported by
%Spain's Ministry of Economy Project MTM2013-43671-P. 
%D. Mora was partially supported by CONICYT-Chile
%through FONDECYT project 1140791 (Chile) and by DIUBB through project 151408 GI/VC,
%Universidad del B\'io-B\'io, (Chile)
%R. Rodr\'iguez was partially supported by BASAL project CMM,
%Universidad de Chile.
%%\end{abstract}


\begin{thebibliography}{9}
%%%%%%%%%%%%%%%%%%%%%%%%%%%

\bibitem{Antonietti}
\textsc{P.~F. Antonietti, A. Buffa, and I. Perugia},
\textit{Discontinuous {G}alerkin approximation of the {L}aplace
              eigenproblem}, 
Comput. Methods Appl. Mech. Engrg., 195 (2006), pp. 3483--3503.


%\bibitem{ABCM}
%\textsc{D.~N. Arnold, F. Brezzi, B. Cockburn, and L.~D. Marini},
%\textit{Unified analysis of discontinuous Galerkin methods for elliptic problems}, 
%SIAM J. Numer. Anal., 5 (2001/02), pp. 1749--1779.

%\bibitem{afw-acta-2006}
%\textsc{D.~N. Arnold, R.~S. Falk, and R. Winther},
%\textit{Finite element exterior calculus, homological techniques, and
%applications}, 
%Acta Numerica, 15 (2006), pp. 1--155.

%\bibitem{afw-2007}
%\textsc{D.~N. Arnold, R.~S. Falk, and R. Winther},
%\textit{Mixed finite element methods for linear elasticity with weakly imposed symmetry}, 
%Math. Comp. 76 (2007),  pp. 1699-1723.

\bibitem{BO}
\textsc{I. Babu\v{s}ka and J. Osborn},
\textit{Eigenvalue Problems}, 
in Handbook of Numerical Analysis, Vol. II, 
P.~G. Ciarlet and J.~L. Lions, eds., 
North-Holland, Amsterdam, 1991, pp.~641--787.

\bibitem{ActaBoffi}
\textsc{D. Boffi}, 
\textit{Finite element approximation of eigenvalue problems},
Acta Numerica, 19, (2010), pp. 1--120.

\bibitem{BoffiBrezziFortinBook}
\textsc{D. Boffi, F. Brezzi, and M. Fortin},
Mixed Finite Element Methods and Applications.
Springer Series in Computational Mathematics, 44. Springer, Heidelberg, 2013.

\bibitem{BoffiBrezziFortin}
\textsc{D. Boffi, F. Brezzi, and M. Fortin},
\textit{Reduced symmetry elements in linear elasticity},
Comm. Pure Appl. Anal., 8 (2009), pp. 1--28.

\bibitem{BoffiBrezziGastaldi(a)}
\textsc{D. Boffi, F. Brezzi, and L. Gastaldi}, 
\textit{On the problem of spurious eigenvalues in the approximation of
linear elliptic problems in mixed form}, 
Math. Comp., 69 (2000), pp.~121--140.


\bibitem{BDM}
\textsc{F. Brezzi, J. Douglas, Jr., and L.D. Marini},
\textit{Two families of mixed finite elements for second order elliptic problems},
Numer. Math., 47 (1985), pp. 217--235.

%\bibitem{BrezziFortin}
%\textsc{F. Brezzi and M. Fortin},
%Mixed and Hybrid Finite Element Methods. 
%Springer Verlag, New York 1991.


%\bibitem{buffaHoustonPerugia}
%{\sc A. Buffa, P. Houston and I. Perugia,}
%{\it Discontinuous Galerkin computation of the Maxwell eigenvalues on simplicial meshes,}
%J. Comput. Appl. Math. 2014 (2007) pp. 317--333.
 
\bibitem{BuffaPerugia}
\textsc{A. Buffa and I. Perugia},
\textit{Discontinuous {G}alerkin approximation of the {M}axwell eigenproblem},
SIAM J. Numer. Anal., 44 (2006), pp. 2198--2226.

%\bibitem{CGS}
%\textsc{B. Cockburn, J. Golapakrishnan, F.J. Sayas}
%\textit{A projection-based error analysys of HDG methods},
%Math. Comp., 79, 271, (2010), pp. 1352--1367.

\bibitem{D} 
{\sc M. Dauge},
Elliptic Boundary Value Problems on Corner Domains,
Lecture Notes in Mathematics, 1341, Springer, Berlin, 1988.

\bibitem{DNR1}
\textsc{J. Descloux, N. Nassif, and J. Rappaz},
\textit{On spectral approximation. Part 1: The problem of convergence},
RAIRO Anal. Num\'er., 12 (1978), pp.~97--112.

\bibitem{DNR2}
\textsc{J. Descloux, N. Nassif, and J. Rappaz},
\textit{On spectral approximation. Part 2: Error estimates for the Galerkin
method},
RAIRO Anal. Num\'er., 12 (1978), pp.~113--119.

%\bibitem{bernardo1}
% {\sc B. Cockburn, J. Gopalakrishnan and J. Guzm\'an},
%  \textit{ A new elasticity element made for enforcing weak stress symmetry}, 
% Math. Comp. 79 (2010),  pp. 1331--1349.

\bibitem{DiPietroErn}
{\sc D.~A. Di Pietro and A. Ern},
Mathematical Aspects of Discontinuous Galerkin Methods,
Springer-Verlag, Berlin Heidelberg, 2012.

%\bibitem{fabes}
%{\sc E.B. Fabes, C.E. Kenig and G.C. Verchota},
%\textit{The Dirichlet problem for the Stokes system on Lipschitz domains},
%Duke Math. J., \textbf{12}, pp. 113--119.

\bibitem{fenics}
 {\sc A. Logg,   K.~A. Mardal,  G.~N. Wells  et al.}
 Automated Solution of Differential Equations by the Finite Element Method, 
 Springer, 2012.
%\bibitem{fenics}
% {\sc A. Logg,   K.-A. Mardal,  G. N. Wells  et al.}
% Automated Solution of Differential Equations by the Finite Element Method. 
% Springer 2012.
 
% \bibitem{gatica1}
% {\sc G.N. Gatica, A. M\'arquez and S. Meddahi},
%  \textit{Analysis of the coupling of Lagrange and Arnold-Falk-Winther 
% finite elements for a fluid-solid interaction problem in three dimensions}, 
%SIAM J. Numer. Anal. 50 (2012), pp. 1648--1674.

\bibitem{girault-raviart}
{\sc V. Girault and P.A. Raviart},
Finite Element Methods for Navier-Stokes Equations. Theory and Algorithms,
Berlin, Springer, (1986).

%\bibitem{gopala}
%{\sc J. Gopalakrishnan, F. Li, N.-C. Nguyen, and J. Peraire},
%\textit{Spectral approximations by the HDG method},
%Math. Comp. 32, 293, (2014),  pp. 1037--1059.


 
 \bibitem{grisvard}
 {\sc P. Grisvard}, 
 \textit{Probl\'ems aux limites dans les polygones. Mode démploi},
 EDF, Bull. Dir. Etudes Rech. Ser. C, 1 (1986), pp. 21--59.

%\bibitem{Hes-Warburton}
%{\sc J.S. Hesthaven and T. Warburton},
%{\it High order nodal discontinuous Galerkin methods for the Maxwell eigenvalue problem,}
%Philos. Trans. R. Soc. Lond. Ser. A Math. Phys. Eng. Sci., 362 (2004), pp. 493--524.

 

 
 
% \bibitem{dominik1}
% {\sc P. Houston, I. Perugia, A. Schneebelia and D. Sch\"{o}tzau},
% \textit{Interior penalty method for the indefinite time-harmonic Maxwell equations},
% Numer. Math. 100 (2005), pp. 485--518.
 
 \bibitem{MMT}
 {\sc A. M\'arquez, S. Meddahi, and T. Tran},
 \textit{Analyses of mixed continuous and discontinuous {G}alerkin methods for the time harmonic 
 elasticity problem with reduced symmetry},
 SIAM J. Sci. Comput., 37 (2015) pp. 1909--1933.
 
\bibitem{MMR}
{\sc S. Meddahi, D. Mora, and R. Rodr\'iguez},
\textit{Finite element spectral analysis for the mixed formulation of
the elasticity equations},
SIAM J. Numer. Anal., 51 (2013) pp. 1041--1063. 
 
\bibitem{MMR-stokes}
 {\sc S. Meddahi, D. Mora, and R. Rodr\'iguez},
\textit{A finite element analysis of a pseudostress formulation for
the Stokes eigenvalue problem},
IMA J. Numer. Anal., 35 (2015) pp. 749--766.
 
 




%\bibitem{SimonReed}
%{\sc M. Reed and B. Simon},
%Methods of modern mathematical physics. I. Functional analysis. Second edition. Academic Press, Inc.  
%[Harcourt Brace Jovanovich, Publishers], New York, 1980. 

%\bibitem{MRR}
%{\sc D. Mora, G. Rivera and R. Rodr\'iguez,}
%\textit{A virtual element method for the Steklov eigenvalue problem}, Math. Models Methods Appl. Sci., 25, (2015), pp. 1421–1445.


%\bibitem{st}
%{\sc R. Stenberg},
%{\it A family of mixed finite elements for the elasticity problem}.
%Numerische Mathematik, 53, (1988), pp. 513--538.
%\bibitem{savare}
%{\sc G. Savar\'e},
%\textit{Regularity results for elliptic equations in Lipschitz domains},
%J. Funct. Anal., 152, 176--201.





%\bibitem{warbourton}
%{\sc T. Warbourton and M. Embree},
%{\it The role of penalty in the local of discontinuous Galerkin method for Maxwell's eigenvalue problem},
%Comput. Methods Appl. Mech. Engrg., 195 (2006), pp.--3205--3223.



\end{thebibliography}
\end{document}